\documentclass[12pt, twoside, english, headsepline]{article}

\usepackage[margin=30mm]{geometry}
\usepackage{scrlayer-scrpage}  
\clearscrheadfoot                 
\cehead[]{F. Cinque}        
\cohead[]{The Negative Reflection Principle and the Telegraph Process}        
\cefoot[\pagemark]{\pagemark}  
\cofoot[\pagemark]{\pagemark}     

\usepackage{sectsty}
\sectionfont{\fontsize{14}{13}\selectfont}
\subsectionfont{\fontsize{12}{13}\selectfont}

\usepackage{tikz,graphicx}
\usepackage{amsmath,amsthm,amssymb,amsfonts, mathtools, amsbsy, dsfont}
\usepackage[utf8]{inputenx}
\usepackage{xcolor}
\usepackage{microtype}
\usepackage{hyperref}
\hypersetup{linkcolor=blue}

\newcommand*\dif{\mathop{}\!\mathrm{d}}  

\allowdisplaybreaks 

\newtheorem{thm}{Theorem}[section] 
\newtheorem{cor} {Corollary}[section] 
\theoremstyle{definition} 
 
\newtheorem{ex}{Example}[section]
\newtheorem{rem}{Remark}[section] 
\numberwithin{equation}{section} 

\providecommand{\keywords}[1]
{
  \small	
  \textbf{\textit{Keywords: }} #1
}
\providecommand{\MSC}[1]
{
  \small	
  \textit{2020 MSC: } #1   
}

\title{The Negative Reflection Principle and the Joint Distribution of the Telegraph Process and its Maximum}
\author{Fabrizio Cinque$^1$\\
        \small Department of Statistical Sciences, Sapienza University of Rome, Italy \\
        \small $^1$fabrizio.cinque@uniroma1.it
}

\date{May 22, 2021} 

\begin{document}

\maketitle

\begin{abstract}
In this paper we study the joint distributions of the telegraph process and its maximum conditioned on the number of changes of direction and the initial velocity. We prove that in the case of positive starting velocity, a form of the reflection principle holds. We call it the negative reflection principle and we generalize it to the paths of the random motions moving with constant finite velocity. In the last section we obtain the conditional distribution of the first passage time of the telegraph process by using the known results on the maximum of the motion. Finally, we derive the distribution of the first returning time to the origin.
\end{abstract} \hspace{10pt}


\MSC{Primary 60G51; 60K99}


\section{Introduction}

The one-dimensional symmetric telegraph process describes the position of a particle starting from the origin $x=0$ at time $t=0$ and moving forwards and backwards on the real line. It moves alternatively with two finite constant velocity $+c$ and $-c$, with $c>0$, and the initial speed is uniformly chosen between the two possible alternatives. The changes of direction are paced by a homogeneous Poisson process $\{N(t)\}_{t\ge0}$ of rate $\lambda>0$, meaning that the displacements of the particle are exponentially distributed with average length $c/\lambda$. By denoting the initial velocity of the motion with $V(0) \sim Unif \{-c, +c\}$, we can define the symmetric telegraph process $\{\mathcal{T}(t)\}_{t\ge0}$ as follows
\begin{equation}\label{INTROdefinizioneProcessoTelegrafo}
\mathcal{T}(t) \coloneqq V(0) \int_0^t (-1)^{N(s)} \dif s = V(0)\sum_{i=0}^{N(t)-1} \Bigl(T_{i+1}-T_{i} \Bigr) (-1)^{i} + V(0)(-1)^{N(t)}\Bigl(t-T_{N(t)}\Bigr)
\end{equation}
where $T_i$ is the $i$-th arrival time of the Poisson process for $i = 1, \dots, N(t)$ and $T_0 = 0 \ a.s.$
\\

The telegraph process is the prototype of the finite speed random motions and it has been formally presented in \cite{G1951}. Several authors have studied the symmetric telegraph process (\ref{INTROdefinizioneProcessoTelegrafo}), see \cite{K1974}, \cite{O1990}, as well as its generalizations. The most common one is the asymmetric motion with velocities $c_1$ and $-c_2,\ 0<c_1\not =c_2>0$, and two possible rates of reversals $\lambda_1\not =\lambda_2>0$, whose probability law was first obtained in \cite{BNO2001} and has been further investigate in many papers, for instance \cite{DcIMZ}. We also recall the motion with Erlang distributed displacements, see \cite{Dc2001}, and the motion describing a particle that uniformly chooses its velocity in the continuous set $[-c,c]$, see \cite{Dg2010}. 
\\
Some researchers also undertook the study of the multidimensional extensions of the telegraph process. Among the most remarkable papers we recall \cite{OK1996}, where the author present a planar motion with orthogonal directions, \cite{KO2005}, regarding the planar random motion with infinite possible directions, and \cite{ODg2007} regarding random motions with finite speed in spaces of dimension $n\ge 3$.
\\

The most relevant feature of the random motions with finite velocity is that they preserve the natural property of moving with finite speed along the same direction, therefore they are suitable to describe real motions and they emerge in several fields. For instance, in finance they can model the stock prices or the volatility of financial markets, see \cite{KR2013}, \cite{R2007}, while in ecology, they model the displacements of wild animals on the soil, see \cite{H1994}. Hence, finite speed random processes represent a realistic alternative to the widely used diffusion processes.
\\

In this paper we will focus on the symmetric telegraph process (\ref{INTROdefinizioneProcessoTelegrafo}) which has the following conditional distributions. Let $x\in (-ct,ct)$ and $v =\pm c$
\begin{equation}\label{pt+-p}  
P \lbrace \mathcal{T}(t) \in \dif x \ |\ V(0)=v,\ N(t) = 2k \rbrace =\frac{(2k)!}{k!(k-1)!} \frac{(c^2t^2-x^2)^{k-1}(ct+sign(v)x)}{(2ct)^{2k}}\dif x,
\end{equation}
for $k\in \mathbb{N}$, and
$$ P \lbrace \mathcal{T}(t) \in \dif x \ |\ V(0)=v,\ N(t) = 2k+1 \rbrace = P \lbrace \mathcal{T}(t) \in \dif x\ |\ N(t) = 2k+1 \rbrace $$
\begin{equation}\label{ptd} 
=P \lbrace \mathcal{T}(t) \in \dif x\ |\ N(t) = 2k+2\rbrace =\frac{(2k+1)!}{k!^2} \frac{(c^2t^2-x^2)^k}{(2ct)^{2k+1}}\dif x,
\end{equation}
for $k\in \mathbb{N}_0$, see \cite{CO2020} and \cite{DgOS2005} for the proof of (\ref{pt+-p}) and (\ref{ptd}).

We also need to recall the fascinating result of \cite{CO2020} concerning the conditional distribution of the maximum of the telegraph process, for $n\in \mathbb{N},\ \beta \in (0,ct)$,
\begin{equation}\label{leggemax+d2}
P\Big\lbrace \max_{0\le s\le t} \mathcal{T}(s) \in \dif \beta\ \Big|\ V(0) = c,\ N(t) = n\Big\rbrace  =2 P \lbrace \mathcal{T}(t) \in \dif \beta \ |\ N(t) = n\rbrace.
\end{equation}

The main result of this paper has been inspired by (\ref{leggemax+d2}) and it concerns the existence of a one-to-one correspondence between the two following sets of trajectories, for $\beta \in [0,ct)$, $x\in (2\beta -ct, \beta]$ and natural $n$,
\begin{equation}\label{INTROPRNinsiememax+}
 \big\lbrace \omega\in \Omega\ :s\mapsto X(\omega, s)\ s.t.\ V(\omega, 0) = +c,\ N(\omega, t) = n,\ \max_{0\le s\le t}X(\omega, s) >\beta,\  X(\omega, t)= x \big\rbrace,
\end{equation}
\begin{equation}\label{INTROPRNinsieme-}
\big\lbrace \omega\in \Omega\ :s\mapsto X(\omega, s)\ s.t.\  V(\omega, 0) = -c,\ N(\omega, t) = n,\ X(\omega, t)= 2\beta-x \big\rbrace,
\end{equation}
where $\{X(t)\}_{t\ge0}$ is a finite constant velocity random motion (meaning the it moves with velocities $\pm c$) whose changes of direction are governed by a point process $\{N(t)\}_{t\ge0}$. We call this property \textit{negative reflection principle}. This resembles the reflection principle, but the trajectories have opposite starting velocities. This means that the trajectories in the first set concerns a motion starting with the positive speed $c$, while the samples in the other set start with the negative velocity $-c$. In Section \ref{sezionePRN} we prove the existence of a bijective relationship between sets (\ref{INTROPRNinsiememax+}) and (\ref{INTROPRNinsieme-}) and we give a graphical method to \textit{negatively reflect} the samples.

In the third section we study the connection between the telegraph process at time $t>0$ and its maximum in the time interval $[0,t]$, $\displaystyle M(t)\coloneqq\max_{0\le s\le t}\mathcal{T}(s)$, conditionally on the initial velocity and the number of changes of direction occurred up to time $t$. 
Throughout the work, we use the following notation for the conditional probability measure, for integer $n\ge0$ and real $t>0$
\begin{equation}\label{notazionemassimo&telegrafo}
P_{n}^\pm\{\ \cdot\ \} \coloneqq P\{\ \cdot\ |\ V(0) = \pm c,\ N(t) = n\}.
\end{equation}
We provide a specific proof of the relationship, for $\beta \in (0, ct)$ and $x \in (2\beta -ct, \beta]$,
\begin{equation}\label{INTROprincipioRiflessioneNegativoTelegrafo}
P_n^+\lbrace M(t) > \beta, \mathcal{T}(t) \in \dif x \rbrace =P_n^-\lbrace \mathcal{T}(t) \in 2\beta-\dif x \rbrace
\end{equation}
with $n\in \mathbb{N}$. We define this formula the \textit{negative reflection principle for the telegraph process}. Again, the important difference with the classical reflection principle is the inversion of the initial velocity.

From the results regarding the positive initial velocity, we obtain the conditional distributions of the motion starting with negative initial velocity. In this case we must consider the possibility that the particle does not overcome the level $x = 0$ in the time interval $[0,t]$. The conditional probability mass of this event has a cyclic behavior and it has been proved in \cite{CO2020}. It reads, for $k\in \mathbb{N}$
\begin{equation*}\label{INTROmassimoSingolarita0}
P\big\{\max_{0\le s\le t}\mathcal{T}(s) = 0\, |\, V(0) = -c,\, N(t) = 2k-1\big\} = \binom{2k}{k}\frac{1}{2^{2k}}=  P\big\{\max_{0\le s\le t}\mathcal{T}(s) = 0\, |\, V(0) = -c,\, N(t) = 2k\big\}.
\end{equation*}
We recall the incredible fact that the previous probability mass is independent of both the current time $t$ and the velocity $c$. By joining the position of the particle at time $t$, $\mathcal{T}(t)$, the probabilities lose this characteristic as well as their cyclic behavior, but we can establish some other interesting relationships. Finally, we exhibit the joint distribution of the telegraph process and its maximum conditioned on a negative starting velocity only.

We point out that the interested reader can easily obtain some other important conditional and unconditional distributions by means of the presented results. For instance, we can immediately derive the conditional distribution of the maximum of the telegraph bridge as well as the conditional distribution of the motion in presence of an absorbing barrier; we recall \cite{O1995} and the recent works \cite{DcMZ2018} and \cite{DcMPZ2020} concerning some further kinds of barriers.

In the last section we focus on the hitting times of the symmetric telegraph process. We describe a new method to study the distribution of the first passage time across a level $\beta>0, \ F_\beta$. The results concerning this random time appeared for the first time in \cite{F1992} and \cite{FK1994}. The authors used the Darlin-Siegert's relationship to prove the distributions of the first passage time conditioned on the starting velocity only. The first passage time of the telegraph process has been investigated in the papers \cite{LR2014}, \cite{M2018}, \cite{SZ2004} and \cite{Z2004}, also for some generalization of the telegraph process. Here, we provide some results on the distribution of $F_\beta$ conditioned on the number of reversals in a time interval $[0,t],\ t>0$, and the starting speed of the motion. We achieve these results by means of the conditional distributions of the maximum of the telegraph process presented in \cite{CO2020}. Finally, we study the conditional and unconditional distribution of the returning time to the origin.

It is interesting to observe that the telegraph process is related to a wider class of processes, the stochastic fluid models. We reference \cite{BOT2005} for some general results for the first passage time, the returning time to the origin and other random times for these models.


\section{The negative reflection principle}\label{sezionePRN}

We begin by considering the positively initially oriented telegraph process and we present the intuition for the negative reflection principle of the telegraph process, which is the starting point of the whole study.

By considering the fundamental relationship (\ref{leggemax+d2}), we can write
\begin{align*}
P_n^+\lbrace \mathcal{T}(t) > \beta \rbrace + P_n^-\lbrace \mathcal{T}(t)> \beta\rbrace &= 2P\lbrace \mathcal{T}(t)> \beta\ |\ N(t) = n\rbrace = P_n^+\lbrace M(t) > \beta \rbrace\\
& = P_n^+\lbrace M(t) > \beta,\  \mathcal{T}(t) > \beta \rbrace + P_n^+\lbrace M(t) > \beta,  \mathcal{T}(t)\le \beta \rbrace,
\end{align*}
for $n \in \mathbb{N},\ \beta \in [0,ct)$, thus
\begin{equation}\label{relazioneMassimo+}
P_n^+\lbrace M(t) > \beta,\  \mathcal{T}(t)\le \beta \rbrace = P_n^-\lbrace \mathcal{T}(t)> \beta\rbrace.
\end{equation}
By conditioning on $N(t) = n$, each path of the telegraph process is uniquely determined by the Poisson times $T_1, ..., T_n$. Their joint random variable is uniformly distributed on the simplex since $N(t)$ is a homogeneous Poisson process. This means that each trajectory of the process has the same probability (density). Thus relation (\ref{relazioneMassimo+}) suggests that for each trajectory of the set
on the right-hand side there exits a trajectory in the set on the left-hand side. With this at hand, it is reasonable to study if there exists a one-to-one correspondence between the following sets of trajectories of the telegraph process,, for $n \in \mathbb{N}$ and $\beta \in [0,ct)$
\begin{equation}\label{insiememax+}
\lbrace \omega\in \Omega\ :\ V(\omega, 0) = +c,\ N(\omega, t) = n,\max_{0\le s\le t} \mathcal{T}(\omega, s) >\beta,\  \mathcal{T}(\omega, t)\le \beta \rbrace,
\end{equation}
\begin{equation}\label{insieme-}
\lbrace \omega\in \Omega\ :\  V(\omega, 0) = -c,\ N(\omega, t) = n,\ \mathcal{T}(\omega, t)> \beta \rbrace,
\end{equation}
with $\bigl(\Omega, \mathcal{F}, P\bigr)$ probability space.
\\

In Theorem \ref{teoremaPRNtelegrafo} we generalize result (\ref{relazioneMassimo+}), by proving that for $\beta \in [0, ct)$ and $x \in (2\beta -ct, \beta]$
\begin{equation}\label{principioRiflessioneNegativoTelegrafo}
P_n^+\lbrace M(t) > \beta, \mathcal{T}(t) \in \dif x \rbrace = P_n^-\lbrace \mathcal{T}(t) \in 2\beta-\dif x \rbrace
\end{equation}
with natural $n\in \mathbb{N}$. We refer to this formula as the \textit{negative reflection principle for the telegraph process}. The adjective \textit{negative} is due to the inversion of the starting speed.  We point out that the classical reflection principle, regarding Wiener process, does not take into account the initial speed of the motion, in fact Brownian motion has both infinite velocity and infinite rate of changes of direction. 
\\

Here, provide a proof of the negative reflection principle for a wide class of finite speed random motions, where the symmetric telegraph process is just a particular case.

\begin{thm}[Negative reflection principle]\label{principioRiflessioneNegativo}
Let $\bigl(\Omega, \mathcal{F}, P\bigr)$ be a probability space. Let $\lbrace X(t) \rbrace_{t\ge 0}$ be a finite velocity random motion such that
\begin{equation}
X(t) = V_0 \int_0^t (-1)^{N(s)}\dif s
\end{equation}
where $V_0\in \{\pm c\}\ a.s., \ c>0$, and $\{N(t)\}_{t\ge 0}$ is an independent point process whose waiting times have a common support $(0,+\infty)$.
\\Let $\beta \in [0,ct)$ and $x\in (2\beta -ct, \beta]$. There exists a bijective relationship between the two following sets of trajectories of the process $X(t)$
\begin{equation}\label{PRNinsiememax+}
\displaystyle\mathcal{P}_n^+ = \lbrace \omega\in \Omega\ : s\mapsto X(\omega, s)\ s.t.\ V_0(\omega) = +c,\ N(\omega, t) = n,\ \max_{0\le s\le t}X(\omega, s) >\beta,\  X(\omega, t)= x \rbrace
\end{equation}
and
\begin{equation}\label{PRNinsieme-}
\mathcal{P}_n^- =\lbrace \omega\in \Omega\ :s\mapsto X(\omega, s)\ s.t.\  V_0(\omega) = -c,\ N(\omega, t) = n,\ X(\omega, t)= 2\beta-x \rbrace
\end{equation}
for $n \in \mathbb{N}$.
\end{thm}

The reader can notice that this bijective relationship holds even if we remove the condition on the number of changes of direction $N(t) = n$ from both the sets in (\ref{PRNinsiememax+}) and (\ref{PRNinsieme-}).

The process $X(t)$ is a non-jumping stochastic motion that moves with two alternating velocities $\pm c$ and starts at $x=0$ at time $t=0$ with an initial speed randomly chosen between the two possible ones. Its changes of direction are governed by the point process $N(t)$.
\\Let $\omega^+\in \mathcal{P}_n^+$ and $f:\mathcal{P}_n^+\longrightarrow \mathcal{P}_n^-$ being the bijective application between sets (\ref{PRNinsiememax+}) and (\ref{PRNinsieme-}). We call $ \omega^-=f(\omega^+) \in \mathcal{P}_n^-$ the \textit{negatively reflected} sample of $\omega^+$ and vice versa.

\begin{proof}
First of all we point out that we require $\beta \in [0,ct),\ x \in (2\beta -ct, \beta]$ otherwise the sets are both negligible.

\begin{figure}
\begin{minipage}{0.5\textwidth}
\centering
\begin{tikzpicture}[scale = 0.71]
\draw [lightgray] (7.3,2.6) -- (0,2.6) node[left, black, scale = 1]{$\beta$}; \draw[dashed, lightgray] (7.3, 2.6) -- (8.4,2.6);
\draw (0,0) -- (0.7, 1.12) -- (1.8, -0.64) -- (4.3,3.36) -- (5.4,1.6) -- (6.3,3.04) -- (7.4, 1.28);
\draw[dashed, gray] (3.825,2.6) -- (3.825,0) node[below, black, scale = 1]{$t_1$};
\filldraw[blue] (3.825,2.6) circle (3pt) node[above left, black, scale = 0.9]{\textbf{A}};
\draw[dashed, gray] (4.775,2.6) -- (4.775,0) node[below, black, scale = 1]{$t_2$};
\filldraw[blue] (4.775,2.6) circle (3pt) node[above right, black, scale = 0.9]{\textbf{B}};
\draw[dashed, gray] (7.4, 1.28) -- (7.4,0) node[below, black, scale = 1]{$t$};
\draw[dashed, gray] (7.4, 1.28) -- (0,1.28) node[left, black, scale = 1]{$x$};
\filldraw[blue]  (7.4, 1.28) circle (3pt) node[right, black, scale = 0.9]{\textbf{C}};
\draw[->, thick] (-1,0) -- (8.4,0) node[below, scale = 1]{$\pmb{s}$};
\draw[->, thick] (0,-1.4) -- (0,5) node[left, scale = 1]{ $\pmb{X(s)}$};
\filldraw[blue] (0,0) circle (3.5pt) node[below left, black, scale =0.9]{\textbf{O}};
\end{tikzpicture}
\caption{Sample path $\omega^+ \in \mathcal{P}_5^+$,  \hspace{1cm}\protect \linebreak with $n=5$.}\label{PRN_g1}
\end{minipage}\hfill
\begin{minipage}{0.5\textwidth}
\centering
\begin{tikzpicture}[scale = 0.71]
\draw [lightgray] (7.3,2.6) -- (0,2.6) node[left, black, scale = 1]{$\beta$}; \draw[dashed, lightgray] (7.3, 2.6) -- (8.4,2.6);
\draw[lightgray, very thin] (0,0) -- (0.7, 1.12) -- (1.8, -0.64) -- (4.3,3.36) -- (5.4,1.6) -- (6.3,3.04) -- (7.4, 1.28);
\draw (0,0) -- (0.475,-0.76)-- (1.65, 1.12) -- (2.75, -0.64)  -- (5.4,3.6) -- (6.3,2.16) -- (7.4, 3.92);
\filldraw[blue] (0.95,0) circle (3pt) node[above, black, scale = 0.9]{\textbf{O'} $\ $};
\draw[blue] (0.95,0) node[below, black, scale = 0.9]{$\ \ \ \ t_2-t_1$};
\draw[dashed, gray] (4.775,2.6) -- (4.775,0) node[below, black, scale = 1]{$t_2$};
\filldraw[blue] (4.775,2.6) circle (3pt) node[above, black, scale = 0.9]{\textbf{B}};
\draw[dashed, gray] (7.4, 3.92) -- (7.4,0) node[below, black, scale = 1]{$t$};
\draw[dashed, gray] (7.4, 3.92) -- (0,3.92) node[left, black, scale = 1]{$2\beta-x$};
\filldraw[blue] (7.4, 3.92) circle (3pt) node[right, black, scale = 0.9]{\textbf{C'}};
\draw[->, thick] (-1,0) -- (8.4,0) node[below, scale = 1]{$\pmb{s}$};
\draw[->, thick] (0,-1.4) -- (0,5) node[left, scale = 1]{ $\pmb{X(s)}$};
\filldraw[blue] (0,0) circle (3.5pt) node[below left, black, scale =0.9]{\textbf{A'}};
\end{tikzpicture}
\caption{$\omega^- \in \mathcal{P}_5^-$, the negatively reflected sample of $\omega^+$.}\label{PRN_g2}
   \end{minipage}
\end{figure}

Let $\omega^+ \in \mathcal{P}_n^+$, then $\exists\ a.s.\ 0< F_\beta^{(1)}(\omega^+)=t_1< F_\beta^{(2)}(\omega^+)=t_2<t$, respectively the first and the second passage time across level $\beta$ of the sample $\{x\, :\, x = X(\omega^+, s),\, s\in [0,t] \}$. The negatively reflected trajectory of $\omega^+$ is graphically obtained as follows (consider Figures \ref{PRN_g1} and \ref{PRN_g2}):
\begin{itemize}
\item[1.] reflect the path from time $t_1$ to time $t_2$ across level $\beta$ (see the polygonal curve from $A$ to $B$ in Figure \ref{PRN_g1});
\item[2.] translate the reflected poly-line in point 1 to the origin of the axes, so it now starts at $(0,0)$ and ends at the point $(t_2-t_1, 0)$ (see the sample from $A'$ to $O'$ in Figure \ref{PRN_g2});
\item[3.] consider the original path in the time interval $[0,t_1]$ and translate it horizontally by the vector $(t_2-t_1,0)$ so it starts at the point $(t_2-t_1, 0)$, that is where the poly-line in steps 1-2 ends, and it ends at $(t_2,\beta)$ (see the poly-lines from $O'$ to $B$ in Figure \ref{PRN_g2});
\item[4.] reflect the (original) path in the time interval $[t_2,t]$ across level $\beta$ (see the broken line from $B$ to $C$ in Figure \ref{PRN_g1} and its counterpart from $B$ to $C'$ in Figure \ref{PRN_g2}).
\end{itemize}
This procedure produces one and only one negatively reflected sample of $\omega^+ \in\mathcal{P}_n^+$. By applying the method in the reverse way, we can build one and only one trajectory in $\mathcal{P}_n^+$ starting from an element $\omega^- \in \mathcal{P}_n^-$. Notice that in $\omega^-$, $t_2-t_1$ represents the first crossing time through level 0 and $t_2$ represents the first crossing time through $\beta$.
Figures \ref{PRN_g1} and \ref{PRN_g2} describe respectively a sample path in $\mathcal{P}_n^+$ and its negatively reflected counterpart in $\mathcal{P}_n^-$, in the case $n=5$.
\\

We now show analytically that the procedure above is bijective.
\\Let integer $n\ge 1$, every trajectory $\omega \in \Omega$ of the process $X(\omega, t)$ can be described by the position of the motion at the $n$ arrival times, $0<T_1(\omega)=t_1 < ... < T_n(\omega)=t_n<t$ of the process $N(t)$. Let $\omega^+ \in \mathcal{P}_n^+$ and let $F_\beta^{(1)}(\omega^+) = t_\beta^{(1)},F_\beta^{(2)}(\omega^+)=t_\beta^{(2)}$, respectively the first and the second passage time across level $\beta$ in the sample $\omega^+$. Let us assume $t_{h-1}<t_\beta^{(1)}< t_{h}$ and $t_{l-1}<t_\beta^{(2)}<t_l$, with $1\le h<l\le n$, meaning that the sample $\omega^+$ crosses $\beta$ for the first time along the $h$-th displacement and it crosses $\beta$ for the second time along the $l$-th displacement. We can uniquely describe $\omega^+$ by means of a $(n+1)$-dimensional vector $v^+ = \phi(\omega^+) = \Bigl(X(\omega^+,t_{i}) = x_{i} : i=1,...,n+1\Bigr)$, where $ t_{n+1} = t$ (we do not strictly need the $(n+1)$-th coordinate since it can be expressed as an affine combination of the other elements, but we maintain it for the sake of clarity). By applying the constructive method above, we obtain the negatively reflected sample $\omega^-\in \mathcal{P}_n^-$ and we can uniquely describe it by means of $v^-=\phi(\omega^-)$. Vectors $v^+, v^-$ are respectively:
\begin{equation}\label{vettori}
v^+ = \begin{pmatrix}
\begin{rcases}
\begin{array}{l}
0<x_1\\
\ \ \ \ \ \ \cdot\\
\ \ \ x_{h-1}
\end{array}
\end{rcases} < \beta \\
\begin{rcases}
\begin{array}{l}
\ \ \ \ x_{h}\\
\ \ \ \ \ \ \cdot\\
\ \ \ x_{l-1}
\end{array}
\end{rcases} > \beta \\
x_{l} <\beta \\\cdot\\\cdot\\x_{n+1}=x <\beta
\end{pmatrix},\ \ \ \ \ \ \
v^- = \begin{pmatrix}
\begin{rcases}
\begin{array}{l}
\ \ \ y_1 = \beta-x_h \\
\ \ \ \ \ \ \ \ \ \ \ \cdot\\
y_{i} = \beta - x_{h-1+i}\\
\ \ \ \ \ \ \ \ \ \ \ \cdot\\
y_{l-h} = \beta -x_{l-1}
\end{array}
\end{rcases} < 0\\
\begin{rcases}
\begin{array}{l}
0<y_{l-h+1} = x_1\\
\ \ \ \ \ \ \ \ \ \ \ \cdot \\
\ \ \ \ \ y_{l-h+j} = x_j\\
\ \ \ \ \ \ \ \ \ \ \ \cdot \\
\ \ \ \ \ y_{l-1} = x_{h-1}
\end{array}
\end{rcases} <\beta\\
y_l = 2\beta -x_{l} >\beta \\\cdot\\\cdot\\y_{n+1} = 2\beta - x> \beta
\end{pmatrix}
\end{equation}
where the elements $y_1,\dots,y_{l-h}$ are consequences of points 1-2 in the building procedure (which work on $x_h,\dots, x_{l-1}$), elements $y_{l-h+1},\dots,y_{l-1}$ are consequences of point 3 (which work on $x_1,\dots, x_{h-1}$) and $y_l,\dots,y_{n+1}$ are consequences of point 4 (which work on $x_l,\dots, x_{n+1}$). Hence, we have the following bijective affine relationship between $v^+$ and $v^-$
\begin{equation}
v^- = f(v^+)= 
\begin{pmatrix}
0\  & \  -I_{l-h}\  &\  0 \\
I_{h-1}\ &\ 0\ & \ 0 \\
0\ &\ 0 \ &\ -I_{n-l+2}
\end{pmatrix}
v^+ +
\begin{pmatrix}
\beta_{l-h}\\
0_{h-1}\\
2\cdot\beta_{n-l+2}
\end{pmatrix}
\end{equation}
where $I_k$ is the $k\times k$ identity matrix, $\beta_k$ and $0_k$ are $k$-dimensional vectors of all $\beta$s and $0$s respectively.
\\It is important to notice that the sample $\omega^+\in \mathcal{P}_n^+$ was characterized by the pair $(h,l)$ which identifies the displacements of the path where the motion crosses level $\beta$ for the first and the second time respectively. By observing the vector representation $v^-$ in (\ref{vettori}) of the negatively reflected path of $\omega^+$, $\omega^- \in\mathcal{P}_n^-$, we have that this trajectory crosses level $0$ for the first time along its $(l-h+1)$-th displacement and crosses level $\beta$ for the first time along its $l$-th displacement, then we can characterize it by means of the pair $(l-h+1,l)$. Now we notice that any sample $\omega^- \in\mathcal{P}_n^-$ with pair $(l-h+1,l),\ 1\le h<l\le n$, can be uniquely expressed in a vector form $v^-=\phi(\omega^-)$ as in (\ref{vettori}) and we can obtain its negatively reflected path $\omega^+ \in\mathcal{P}_n^+$, with pair $(h,l)$, by means of $f^{-1}$. Hence, we have proved the existence of a bijective relationship between the two following sets
\begin{equation}
\mathcal{P}_n^+ \cap \{\omega\in \Omega\ :\ T_{h-1}(\omega)<F_\beta^{(1)}(\omega)<T_{h}(\omega)\ ,\ T_{l-1}(\omega)<F_\beta^{(2)}(\omega)<T_{l}(\omega)\}
\end{equation}
and 
\begin{equation}
\mathcal{P}_n^- \cap \{\omega\in \Omega\ :\ T_{l-h}(\omega)<F_{0}^{(1)}(\omega)<T_{l-h+1}(\omega)\ ,\ T_{l-1}(\omega)<F_\beta^{(1)}(\omega)<T_{l}(\omega)\}
\end{equation}
where $F_0^{(1)}$ is the first returning time to $0$.
Finally, we simply observe that the function $g(h,l) = (l-h+1, l)$ is an automorphism. Thus, any negatively reflected trajectory of a sample $\omega \in \mathcal{P}_n^+$, with  pair $(h,l)$, is different from the negatively reflected path of a sample $\omega' \in \mathcal{P}_n^+$ with pair $(h',l') \not = (h,l)$. This concludes the proof of the theorem.
\end{proof}

We point out that in building a transformation between trajectories of a finite velocity stochastic motion, it is important not to create new constraints on the paths. For instance, to force a change of direction at a specific level is a serious mistake.
\\

It is important to keep in mind that the negative reflection principle presented in Theorem \ref{principioRiflessioneNegativo} does not consider probabilities, it regards trajectories only. Result (\ref{principioRiflessioneNegativoTelegrafo}), that we prove in Theorem \ref{teoremaPRNtelegrafo}, regarding the probabilities, is due to the uniform distribution of the arrival times in the case of the homogeneous Poisson process which implies that the trajectories have all the same probability (density).

The negative reflection principle holds for both continuous and discrete processes (it is sufficient to suitably adapt the point process $N(t)\,$ and consider $X(t)$ at the desired times only). For instance, the reader can prove that it holds in the case of simple random walks. In particular, for a simple symmetric random walk we can obtain an equality in terms of the probabilities since the sample paths of a symmetric random walk have all the same probability.

\begin{ex}[Simple symmetric random walk]
Let $S_0 = 0\ a.s.$ and $S_n = S_0+\sum_{i=1}^n X_i$, with $X_i$ independent uniformly distributed random variables in $\lbrace -1, 1\rbrace$. Let $n\ge 2$, for $\beta \in \lbrace 0,1,...,n-2\rbrace$ and $x\in\lbrace 2(\beta+1) -n,...,\beta\rbrace$ we have that
\begin{align*}
P\lbrace \max_{0\le k \le n} S_k > \beta, S_n = x\ |\ S_1 = 1\rbrace &= P\lbrace \max_{0\le k \le n-1} S_k  > \beta-1, S_{n-1} = x-1 \rbrace \\
& = P\lbrace S_{n-1} = 2\beta-x+1\rbrace = P\lbrace S_n = 2\beta-x\ |\ S_1 = -1\rbrace,
\end{align*}
where in the second equality we used the known result on the maximum of a simple symmetric random walk, due to the classical reflection principle. The sets in the first and the fourth member respectively correspond to the sets (\ref{PRNinsiememax+}) and (\ref{PRNinsieme-}), with no condition on the point process $N(t)$, in the case of random walks.
\hfill $\diamond$
\end{ex}

\section{Joint distribution of the symmetric telegraph process and its maximum}

In this section we present the conditional and unconditional joint distributions of the position of the telegraph particle at time $t>0$ and its maximum in the time interval $[0,t]$. The results we present here, easily yield other conditional distributions of the symmetric telegraph process.

\subsection{Motion starting with positive velocity}

We now give a direct proof of the negative reflection principle for the telegraph process.

\begin{thm}[Negative reflection principle for the telegraph process]\label{teoremaPRNtelegrafo}
Let $\lbrace \mathcal{T}(t) \rbrace_{t\ge0}$ be a symmetric telegraph process. Let $n\in \mathbb{N}$ and $x \in (-ct,ct)$
\begin{equation}\label{principioRiflessioneNegativoTelegrafo1.1}
P_n^+\lbrace M(t) \le \beta,\ \mathcal{T}(t) \in \dif x \rbrace = \begin{cases}\begin{array}{l l}
0 &  if \ \beta <\max\{0,x\},\\
P_n^+\lbrace \mathcal{T}(t) \in \dif x \rbrace -P_n^-\lbrace \mathcal{T}(t) \in 2\beta-\dif x \rbrace &  if \ \max\{0,x\}\le\beta<\frac{ct+x}{2},\\
P_n^+\lbrace \mathcal{T}(t) \in \dif x \rbrace &  if \ \beta \ge\frac{ct+x}{2}.
\end{array}\end{cases}
\end{equation}
\end{thm}

\begin{proof}
When $N(t) =n=1$ distribution (\ref{principioRiflessioneNegativoTelegrafo1.1}) reduces to
\begin{equation}\label{principioRiflessioneNegativoCaso1}
P_1^+\lbrace M(t) \le \beta,\ \mathcal{T}(t) \in \dif x \rbrace = \begin{cases}\begin{array}{l l}
0 & \ if \ \beta <\frac{ct+x}{2},\\
P_1^+\lbrace \mathcal{T}(t) \in \dif x \rbrace=\displaystyle\frac{\dif x}{2ct} & \ if \ \beta \ge\frac{ct+x}{2},
\end{array}\end{cases}
\end{equation}
since the random variable $\mathcal{T}(t)$ is uniformly distributed in $(-ct,ct)$ if one Poisson event occurs in the time interval $[0,t]$, see (\ref{ptd}) with $k=0$. To prove (\ref{principioRiflessioneNegativoCaso1}) we observe that the process reaches level $\beta$ at time $\frac{\beta}{c}$, then it changes direction and it keeps moving with speed $-c$ until time $t$ where it will be located in $\beta - c(t-\frac{\beta}{c}) = 2\beta -ct = x$. Thus, $P_1^+\Big\lbrace M(t) =\frac{ct+\mathcal{T}(t)}{2},\  \mathcal{T}(t) \in \dif x \Big\rbrace = P_1^+\{ \mathcal{T}(t) \in \dif x\}$.

The first and third cases of (\ref{principioRiflessioneNegativoTelegrafo1.1}) are trivial for $n\ge2$. We focus on the second case, which can be written as follows, for $\beta \in [0, ct)$ and $x\in (2\beta-ct, \beta]$,
\begin{equation}\label{principioRiflessioneNegativoTelegrafo2}
P_n^+\lbrace M(t) \le \beta,\ \mathcal{T}(t) \in \dif x \rbrace = P_n^+\lbrace \mathcal{T}(t) \in \dif x \rbrace -P_n^-\lbrace \mathcal{T}(t) \in 2\beta-\dif x \rbrace.
\end{equation}
For $n=2$, denoting with $T_j$ the $j$-th Poisson arrival time, we consider that
\begin{equation*}\label{baseinduttiva2}
P^+_2\{ M(t) \le \beta,\ \mathcal{T}(t)\in \dif x\} =\int_0^{\frac{\beta}{c}} P\Big\{ T_1 \in  \dif t_1,\ T_2 \in \frac{ct-\dif x}{2c}+T_1\ |\ N(t) = 2\Big\} = \frac{\beta }{c^2t^2}\dif x
\end{equation*}
which coincides with result (\ref{principioRiflessioneNegativoTelegrafo2}) when $n=2$.

We now prove (\ref{principioRiflessioneNegativoTelegrafo2}) for all natural $n>2$ by means of an induction argument. The case above, $n=2$, represents the induction base for $n$ even.
Let natural $n\ge 3$, at time $T_2 = t_2<t$:

- Case 1. The motion has enough time to cross level $\beta$ and reach $x$ at time $t$, so if $\\c(t-t_2)\ge (\beta-2ct_1+ct_2)+(\beta-x) \iff t_2\le \frac{ct-2\beta+x}{2c}+t_1$.

- Case 2. The motion has not enough time to cross level $\beta$, but it has time to reach $x$ at time $t$, then if $t_2> \frac{ct-2\beta+x}{2c}+t_1$ and $c(t-t_2)\ge |x-2ct_1+ct_2|$.
\\Thus we can write the following recurrence relationship, for natural $n\ge3$
\begin{align}
&P_n^+\lbrace M(t) \le \beta, \mathcal{T}(t) \in \dif x \rbrace =\label{integrale_PRNinduzione1}\\
& = \int_0^{\frac{\beta}{c}}\int_{t_1}^{\frac{ct-2\beta+x}{2c}+t_1} P_{n-2}^+\lbrace M(t-t_2) \le \beta-2ct_1+ct_2, \mathcal{T}(t-t_2) \in \dif x-2ct_1+ct_2 \rbrace \nonumber\\
&\ \ \ \times P\{T_1 \in \dif t_1, T_2\in \dif t_2 | N(t) = n\} \nonumber\\
&\ \ \ +\int_0^{\frac{\beta}{c}}\int_{\frac{ct-2\beta+x}{2c}+t_1}^{\frac{ct-x}{2c}+t_1} P_{n-2}^+\lbrace \mathcal{T}(t-t_2) \in \dif x-2ct_1+ct_2 \rbrace  P\{T_1 \in \dif t_1,\ T_2\in \dif t_2\ |\ N(t) = n\}\nonumber
\end{align}
Let $k\in \mathbb{N}$, we assume that distribution (\ref{principioRiflessioneNegativoTelegrafo2}) holds for $n = 2k$ (induction hypothesis). Hence, for $n=2k+2$, formula (\ref{integrale_PRNinduzione1}) reads
\begin{align*}
P_{2k+2}^+&\lbrace M(t) \le \beta, \mathcal{T}(t) \in \dif x \rbrace \\
&=\int_0^{\frac{\beta}{c}}\int_{t_1}^{\frac{ct-x}{2c}+t_1} P_{2k}^+\lbrace \mathcal{T}(t-t_2) \in \dif x-2ct_1+ct_2 \rbrace  P\{T_1 \in \dif t_1,\ T_2\in \dif t_2\ |\ N(t) = 2k+2\}\\
&\ \ \ - \int_0^{\frac{\beta}{c}}\int_{t_1}^{\frac{ct-2\beta+x}{2c}+t_1} P_{2k}^-\lbrace \mathcal{T}(t-t_2) \in 2\beta -\dif x-2ct_1+ct_2 \rbrace P\{T_1 \in \dif t_1,\ T_2\in \dif t_2\ |\ N(t) = 2k+2\} \\
& = \dif x \int_0^{\frac{\beta}{c}} \dif t_1\int_{t_1}^{\frac{ct-x}{2c}+t_1} \frac{(ct-x+2ct_1-2ct_2)^{k-1}(ct+x-2ct_1)^{k}}{\bigl[2c(t-t_2)\bigr]^{2k}}\frac{(2k+2)!}{k!(k-1)!}\frac{(t-t_2)^{2k+2}}{t^{2k+2}}\dif t_2 \\
& \ \ \ -\dif x \frac{(2k+2)!}{k!(k-1)!}\int_0^{\frac{\beta}{c}} \dif t_1\int_{t_1}^{\frac{ct-2\beta+x}{2c}+t_1} \frac{(ct-2\beta+x+2ct_1-2ct_2)^{k}(ct+2\beta-x-2ct_1)^{k-1}}{(2c)^{2k}t^{2k+2}}\dif t_2\\
&= \frac{(2k+2)!}{(k+1)!\,k!} \frac{\dif x}{(2ct)^{2k+2}}\Bigl[ (c^2t^2-x^2)^k(ct+x) - \bigl[c^2t^2-(2\beta-x)^2\bigr]^k \bigl[ct-(2\beta-x) \bigr]\Bigr] \\
& = P_{2k+2}^+\lbrace \mathcal{T}(t) \in \dif x \rbrace -P_{2k+2}^-\lbrace \mathcal{T}(t) \in 2\beta-\dif x \rbrace
\end{align*}
which concludes the proof of the Theorem for $n$ even. 

In the case of $n$ odd the proof works in the same way and therefore it is omitted. We note only that in formula (\ref{integrale_PRNinduzione1}) for $n=3$, by considering (\ref{principioRiflessioneNegativoCaso1}), the first term is equal to $0$ and it is trivial to see that (\ref{principioRiflessioneNegativoTelegrafo2}) holds when $n=3$. This is the induction base in the case of $n$ odd.
\end{proof}

Thanks to the negative reflection principle for the telegraph process (\ref{principioRiflessioneNegativoTelegrafo2}), we easily achieve the joint density of the telegraph process and its maximum as we show in the next corollaries. 

\begin{cor}
Let $\lbrace \mathcal{T}(t) \rbrace_{t\ge0}$ be a symmetric telegraph process. Let $\beta \in (0, ct), \ x\in (2\beta-ct, \beta)$
\begin{equation}\label{met+d}
P_{2k+1}^+\lbrace M(t) \in \dif \beta,\ \mathcal{T}(t) \in \dif x \rbrace = \frac{(2k+1)!}{k!(k-1)!}\frac{(2\beta -x)\bigl(c^2t^2-(2\beta -x)^2 \bigr)^{k-1}}{2^{2k-1}(ct)^{2k+1}}\dif \beta \dif x,
\end{equation}
for $k \in \mathbb{N}$. For $|x|<ct$
\begin{equation}\label{met+1}
P_1^+\lbrace M(t) =\frac{ct+\mathcal{T}(t)}{2},\  \mathcal{T}(t) \in \dif x \rbrace =P_1^+\lbrace \mathcal{T}(t) \in \dif x \rbrace,
\end{equation}
or alternatively, for $\beta \in (0,ct)$
\begin{equation}\label{met+1_2}
P_1^+\lbrace M(t) \in \dif \beta,\  \mathcal{T}(t)=2M(t)-ct \rbrace =P_1^+\lbrace M(t) \in \dif \beta \rbrace.
\end{equation}
\end{cor}

\begin{proof}
For natural $n\ge 2, \ \beta \in (0,ct)$ and $x\in (2\beta -ct, \beta)$
\begin{equation}\label{relazionemet+}P_n^+\lbrace M(t) \in \dif \beta,\ \mathcal{T}(t) \in \dif x \rbrace / \dif \beta  = \frac{\partial }{\partial \beta} P_n^+\lbrace M(t) \le \beta, \mathcal{T}(t) \in \dif x \rbrace =-\frac{\partial }{\partial \beta}P_n^-\lbrace \mathcal{T}(t) \in 2\beta -\dif x\rbrace 
\end{equation}
where in the last equation we used the negative reflection principle for the telegraph process (\ref{principioRiflessioneNegativoTelegrafo2}).
Formula  (\ref{met+d}) follows from (\ref{relazionemet+}) and (\ref{ptd}) by performing simple calculations.

Result (\ref{met+1}) was shown at the beginning of the proof of Theorem \ref{teoremaPRNtelegrafo}.
\end{proof}

\begin{rem}
Let integer $k\ge 1, \ \beta \in (0,ct)$ and $2\beta -ct<x\le\beta$
\begin{align*}
P_{2k+1}^+&\lbrace M(t) \in \dif \beta,\ \mathcal{T}(t) < x \rbrace = \dif \beta \int_{2\beta-ct}^x \frac{(2k+1)!}{k!(k-1)!}\frac{(2\beta -y)\bigl(c^2t^2-(2\beta -y)^2 \bigr)^{k-1}}{2^{2k-1}(ct)^{2k+1}}\dif y \\
&= \frac{(2k+1)!}{k!^2}\frac{\bigl(c^2t^2-(2\beta -x)^2 \bigr)^{k}}{2^{2k}(ct)^{2k+1}} \dif \beta = P_{2k+1}^+\lbrace \mathcal{T} (t) \in 2\dif \beta -x \rbrace = \frac{P_{2k+1}^+\lbrace M (t) \in 2\dif \beta -x \rbrace}{2}
\end{align*}
which for $x = \beta$ coincides with distribution (\ref{leggemax+d2}). \hfill $\diamond$
\end{rem}

We now give the details about the case when $N(t) = 2k$ for natural $k$. In this case we must consider that the motion moves with positive speed at time $t$ and it may occur that it reaches its maximum at the end of the time interval. Therefore it can happen that the position at time $t$ coincides with the maximum in $[0,t]$.

\begin{cor}\label{teoremamet+p}
Let $\lbrace \mathcal{T}(t) \rbrace_{t\ge0}$ be a symmetric telegraph process. Let $k \in \mathbb{N}$
\begin{equation}\label{met+p}
 P_{2k}^+\lbrace M(t) \in \dif \beta,\ \mathcal{T}(t) \in \dif x \rbrace = \frac{(2k)!}{k!(k-1)!}\frac{\bigl[ct-(2\beta -x) \bigr]^{k-1}\bigl[ct+(2\beta -x) \bigr]^{k-2}}{2^{2k-1}(ct)^{2k}}\bigl[ ct+(2k-1)(2\beta -x)\bigr] \dif \beta \dif x,
\end{equation}
for $\beta \in (0, ct), \ x\in (2\beta-ct, \beta)$, and
\begin{align}
P_{2k}^+\lbrace M(t) =\mathcal{T}(t) \in \dif \beta \rbrace & = P_{2k}^+\lbrace \mathcal{T}(t) \in \dif \beta \rbrace - P_{2k}^-\lbrace \mathcal{T}(t) \in \dif \beta \rbrace\label{rel_m=t+_t}\\
&= \frac{(2k)!}{k!(k-1)!}\frac{2\beta(c^2t^2-\beta^2)^{k-1}}{(2ct)^{2k}} \dif \beta\label{m=t+p},
\end{align}
for $\beta \in (0,ct)$.
\end{cor}

\begin{proof}
Thesis (\ref{met+p}) follows from (\ref{relazionemet+}) and (\ref{pt+-p})
\begin{align}
 P_{2k}^+&\lbrace M(t) \in \dif \beta, \ \mathcal{T}(t) \in \dif x \rbrace/\dif \beta = -\frac{\partial }{\partial \beta}P_{2k}^-\lbrace \mathcal{T} (t) \in 2\beta -\dif x\rbrace \nonumber\\
& = \frac{(2k)!}{k!(k-1)!} \Biggl[\frac{k\bigl[c^2t^2-(2\beta -x)^2 \bigr]^{k-1}-(k-1)\bigl[ct-(2\beta -x) \bigr]^k\bigl[ct+(2\beta -x) \bigr]^{k-2}}{2^{2k-1}(ct)^{2k}} \Biggr] \dif x\label{met+pcomoda}
\end{align}
which coincides with (\ref{met+p}).
\begin{align*}
P_{2k}^+\lbrace M(t) =\mathcal{T}(t) \in \dif \beta \rbrace &= P_{2k}^+\lbrace \mathcal{T}(t) \in \dif \beta \rbrace -P_{2k}^+\lbrace M(t) >\beta, \mathcal{T}(t) \in \dif \beta \rbrace \\
&= P_{2k}^+\lbrace \mathcal{T}(t) \in \dif \beta \rbrace - P_{2k}^-\lbrace \mathcal{T}(t) \in \dif \beta \rbrace,
\end{align*}
where in the last equation we used the negative reflection principle for the telegraph process (\ref{principioRiflessioneNegativoTelegrafo}). Alternatively we can simply use (\ref{principioRiflessioneNegativoTelegrafo2}) with $x = \beta$.
\end{proof}

We point out that by comparing probabilities (\ref{rel_m=t+_t}) and (\ref{tppN+t_relazione}) we obtain that
\begin{align}
P_{2k}^+\lbrace M(t) =\mathcal{T}(t) \in \dif \beta \rbrace  &= \frac{\dif \beta}{c \, \dif t}P\{F_\beta\in \dif t\ |\ V(0) = c,\ N(t) = 2k\}\label{m=t_tpp+p}\\
& = \frac{\beta}{ct}P\{M(t)\in \dif \beta \ |\ V(0) = c,\ N(t) = 2k\}.\nonumber
\end{align}
By keeping in mind the result of \cite{CO2020} concerning the expected value of the maximum of the telegraph process, that is, for integer $k\ge 1$, $\mathbb{E}[ M(t)\ |\ V(0) = c,\ N(t) = 2k] = \binom{2k}{k}ct/2^{2k}$, we easily obtain the conditional probability that the maximum in $[0,t]$ is equal to the position of the motion at time $t$
\begin{align}
P\lbrace M(t) =\mathcal{T}(t)\ |\ V(0) = c,\ N(t) = 2k\} &= \int_0^{ct} P_{2k}^+\lbrace M(t) =\mathcal{T}(t) \in \dif \beta \rbrace \label{relazonione_met_m0}\\
&= \binom{2k}{k}\frac{1}{2^{2k}} = P\{M(t) =0\ |\ V(0) = -c,\ N(t) = 2k\}.\nonumber
\end{align}
Formula (\ref{relazonione_met_m0}) shows that the probability of the particle, moving with positive initial velocity, of being at time $t$ at the maximum level it reached in the interval $[0,t]$ is equal to the probability that the motion, starting with negative speed, never crosses the level $0$ before time $t$. We point out that probability (\ref{relazonione_met_m0}) is independent of both the current time $t$ and the velocity $c$.

\begin{cor}\label{corollariomet+pripx}
Let $\lbrace \mathcal{T}(t) \rbrace_{t\ge0}$ be a symmetric telegraph process. Let $k \in \mathbb{N},\ \beta \in (0, ct),\ x \in (2\beta-ct, \beta]$
\begin{align}
P_{2k}^+\lbrace M(t) \in \dif \beta, \ \mathcal{T}(t) < x \rbrace &= P_{2k}^-\lbrace \mathcal{T}(t) \in 2\dif \beta-x \rbrace = \frac{P_{2k}^-\lbrace M(t) \in 2\dif \beta-x \rbrace}{2} \label{rel_met<+_t}\\
&=  \frac{2(2k)! }{k!(k-1)!}  \frac{\bigl[ct-(2\beta -x) \bigr]^k \bigl[ct+(2\beta -x) \bigr]^{k-1}}{(2ct)^{2k}} \dif \beta.\nonumber
\end{align}
\end{cor}

\begin{proof}
By means of formula (\ref{met+pcomoda}), we have
\begin{align*}
P_{2k}^+&\lbrace M(t) \in \dif \beta, \ \mathcal{T}(t) < x \rbrace \\
&=\frac{2(2k)! \ \dif \beta}{k!(k-1)!(2ct)^{2k}}\int_{2\beta-ct}^{x} \biggl(k\bigl[c^2t^2-(2\beta -y)^2 \bigr]^{k-1}-(k-1)\bigl[ct-(2\beta -y) \bigr]^k\bigl[ct+(2\beta -y) \bigr]^{k-2} \biggr)\dif y \\
&=\frac{2(2k)! \ \dif \beta}{k!(k-1)!(2ct)^{2k}}\Biggl[k \int_{2\beta-ct}^{x}\bigl[c^2t^2-(2\beta -y)^2 \bigr]^{k-1} \dif y + \bigl[ct-(2\beta -x) \bigr]^k \bigl[ct+(2\beta -x) \bigr]^{k-1} \\
&\ \ \ - k \int_{2\beta-ct}^{x} \bigl[ct-(2\beta -y) \bigr]^{k-1}\bigl[ct+(2\beta -y) \bigr]^{k-1} \dif y\Biggr],
\end{align*}
which coincides with (\ref{rel_met<+_t}). Note that in the last equality we integrated by parts the second integral.
\end{proof}

By means of (\ref{rel_met<+_t}) we can easily check the following
\begin{align*}
P_{2k}^+\lbrace M(t) \in \dif \beta, \ \mathcal{T}(t) \le \beta \rbrace &= P_{2k}^+\lbrace M(t) \in \dif \beta, \ \mathcal{T}(t) < \beta \rbrace + P_{2k}^+\lbrace M(t)=\mathcal{T}(t)  \in \dif \beta \rbrace \\
& = 2P_{2k}^-\lbrace \mathcal{T}(t) \in \dif \beta\rbrace +P_{2k}^+\lbrace \mathcal{T}(t) \in \dif \beta \rbrace - P_{2k}^-\lbrace \mathcal{T}(t) \in \dif \beta \rbrace = P_{2k}^+\lbrace M(t) \in \dif \beta \rbrace,
\end{align*}
where we used (\ref{leggemax+d2}) in the last equation.

\begin{rem}
In light of (\ref{rel_met<+_t}), we can write
\begin{align}
P_{2k}^+\lbrace M(t)= \mathcal{T}(t)\in \dif \beta \rbrace &= P_{2k}^+\lbrace M(t) \in \dif \beta \rbrace - P_{2k}^+\lbrace M(t) \in \dif \beta, \ \mathcal{T}(t)<\beta \rbrace \nonumber\\
& = P_{2k}^+\lbrace M(t) \in \dif \beta\rbrace -P_{2k}^-\lbrace M(t) \in \dif \beta\rbrace,
\end{align}
that is another form of (\ref{rel_m=t+_t}).
\\Furthermore, we have that, for $\beta \in (0, ct)$ and $0<x<\beta$
\begin{align}
\int_x^\beta P_{2k}^+\lbrace M(t) \in \dif w,\ \mathcal{T}(t) \in \dif x \rbrace& = P_{2k}^+\lbrace M(t) \le\beta,\ \mathcal{T}(t)\in \dif x \rbrace -P_{2k}^+\lbrace M(t)= \mathcal{T}(t)\in \dif x \rbrace \nonumber\\
&= P_{2k}^-\lbrace \mathcal{T}(t) \in \dif x \rbrace -P_{2k}^-\lbrace \mathcal{T}(t) \in 2\beta-\dif x \rbrace,\label{met+pint}
\end{align}
where in the last equation we used formulas (\ref{principioRiflessioneNegativoTelegrafo2}) and (\ref{rel_m=t+_t}).
\hfill $\diamond$
\end{rem}

It is important to establish these relationships in order to show the strong connection between all the distributions. Furthermore, we are often able to avoid the use of the explicit form of the joint distributions in the calculations.

Now we present the joint distribution of the motion and its maximum conditioned on a positive starting velocity.

\begin{thm}
Let $\{\mathcal{T}(t)\}_{t\ge0}$ be a symmetric telegraph process. For $\beta\in (0,ct)$
\begin{equation}\label{met=f(m)+}
P\Big\{M(t) \in \dif \beta,\ \mathcal{T}(t) = 2M(t)-ct\ \Big|\ V(0)=c\Big\}= \frac{\lambda\,e^{-\lambda t}}{c}\dif \beta,
\end{equation}
or alternatively, for $x\in (-ct, ct)$
\begin{equation}\label{m=f(t)et+}
P\Big\{M(t) =\frac{\mathcal{T}(t)+ct}{2},\ \mathcal{T}(t)\in \dif x\ \Big|\ V(0)=c\Big\}= \frac{\lambda\,e^{-\lambda t}}{2c}\dif x.
\end{equation}
For $\beta \in (0,ct),\ x\in(2\beta-ct,\beta)$
\begin{align}
P\{M(t) \in \dif \beta,\ \mathcal{T}&(t)\in \dif x\ |\ V(0)=c\}\label{met+}\\
& = \frac{\lambda\,e^{-\lambda t}}{c\sqrt{ct+(2\beta-x)}}\Biggl[ \frac{\lambda(2\beta-x)}{c\sqrt{ct+(2\beta-x)}}\, I_0\Bigl(\frac{\lambda}{c}\sqrt{c^2t^2-(2\beta-x)^2}\Bigr)\nonumber\\
&\ \ \ + \biggl(\frac{\lambda (2\beta-x)}{c\sqrt{ct-(2\beta-x)}}+\frac{\sqrt{ct-(2\beta-x)}}{ct+(2\beta-x)}\biggr)\,I_1\Bigl(\frac{\lambda}{c}\sqrt{c^2t^2-(2\beta-x)^2}\Bigr)\Biggr] \dif \beta \dif x\nonumber
\end{align}
and
\begin{equation}\label{m=t+}
P\{M(t) = \mathcal{T}(t)\in \dif \beta\ |\ V(0)=c\}= \frac{\lambda \beta\,e^{-\lambda t}}{c\sqrt{c^2t^2-\beta^2}}\,I_1\Bigl(\frac{\lambda}{c}\sqrt{c^2t^2-\beta^2}\Bigr) \dif \beta.
\end{equation}
Finally, there is a singular component
\begin{equation}\label{m=t=ct+}
P\{M(t) = \mathcal{T}(t)=ct\ |\ V(0)=c\}= e^{-\lambda t}.
\end{equation}
\end{thm}
As usual, $I_r(x) = \sum_{j=0}^\infty \bigl(\frac{x}{2}\bigr)^{2j+r} \frac{1}{j!\Gamma(j+1+r)}$ is the modified Bessel function of first type with imaginary argument and order $r\in \mathbb{R}$.

\begin{proof}
Formulas (\ref{met=f(m)+}) and (\ref{m=f(t)et+}) can be respectively derived from (\ref{met+1}) and (\ref{met+1_2}) where only one change of direction occurs in the time interval $[0,t]$.

To prove result (\ref{met+}) we consider the following joint distributions. For $\beta \in (0,ct)$ and $x\in (2\beta-ct, \beta)$, by means of (\ref{met+d}) we obtain that
\begin{equation}\label{met+Dispari}
P\{M(t) \in \dif \beta,\, \mathcal{T}(t)\in \dif x,\, N(t) \ \text{odd}\, |\, V(0)=c\}=\frac{\lambda^2\,(2\beta-x)\,e^{-\lambda t}}{c^2\,\sqrt{c^2t^2-(2\beta-x)^2}}\,I_1\Bigl(\frac{\lambda}{c}\sqrt{c^2t^2-(2\beta-x)^2}\Bigr)\dif \beta \dif x
\end{equation}
and by means of (\ref{met+pcomoda}) we obtain that
\begin{align}
P&\{M(t) \in \dif \beta,\ \mathcal{T}(t)\in \dif x,\ N(t) \ \text{odd}\ |\ V(0)=c\}/(\dif \beta \dif x)\nonumber\\
& =\frac{\lambda^2\,e^{-\lambda t}}{2c^2} \Biggl[I_0\Bigl(\frac{\lambda}{c}\sqrt{c^2t^2-(2\beta-x)^2}\Bigr)-\frac{ct-(2\beta-x)}{ct+(2\beta-x)}\,I_2\Bigl(\frac{\lambda}{c}\sqrt{c^2t^2-(2\beta-x)^2}\Bigr) \Biggr]\nonumber\\
&= \frac{\lambda\,e^{-\lambda t}}{c\bigl[ ct+(2\beta-x)\bigr]}\, \Biggl[\frac{\lambda\,(2\beta-x)}{c}\,I_0\Bigl(\frac{\lambda}{c}\sqrt{c^2t^2-(2\beta-x)^2}\Bigr)+ \sqrt{\frac{ct-(2\beta-x)}{ct+(2\beta-x)}}\,I_1\Bigl(\frac{\lambda}{c}\sqrt{c^2t^2-(2\beta-x)^2}\Bigr)\Biggr], \label{met+Pari2}
\end{align}
where in the last equality we used the recurrence relationship for the Bessel functions, with $n$ integer
\begin{equation}\label{relazioneBessel}
I_{n+1}(x)= I_{n-1}(x)-\frac{2n}{x}I_n(x).
\end{equation}
By summing up (\ref{met+Dispari}) and (\ref{met+Pari2}) we obtain the claimed result (\ref{met+}).

To prove (\ref{m=t+}) it is sufficient to consider (\ref{m=t+p}) and perform some simple calculation.

Finally, the probability mass (\ref{m=t=ct+}) coincides with the probability that no switches occur in $[0,t]$.
\end{proof}

We point out that, as well as we showed in the conditional case (\ref{m=t_tpp+p}), probabilities (\ref{m=t+}) and (\ref{tpp+}) are related by
\begin{equation}
P\{M(t) = \mathcal{T}(t)\in \dif \beta\ |\ V(0) = c\} =  \frac{\dif \beta}{c \, \dif t}P\{F_\beta\in \dif t\ |\ V(0) = c\}.
\end{equation}

\subsection{Motion starting with negative velocity}

We now study the case when the motion starts with a negative initial velocity. It is important to recall that, when $V(0)=-c$, the particle may not cross the level $x=0$ in the time interval $[0,t]$ with a positive probability, displayed in formula (\ref{INTROmassimoSingolarita0}).
\\

We begin by recalling the result of Theorem $3.1$ of \cite{CO2020}, for $\beta \in (0,ct)$ and natural $k$, we have
\begin{equation}\label{leggemax-p}
P_{2k}^-\{M(t)\in \dif \beta \} =  \frac{2(2k)!}{k!(k-1)!}\frac{(ct-\beta)^{k}(ct+\beta)^{k-1}}{(2ct)^{2k}}\dif \beta= 2\, P_{2k}^-\{\mathcal{T}(t)\in \dif \beta\}.
\end{equation}
This distribution already shows that a classical reflection principle holds. The next results are confirming this incredible behavior.

\begin{thm}\label{teoremamet-p}
Let $\lbrace \mathcal{T}(t) \rbrace_{t\ge0}$ be a symmetric telegraph process. Let $k \in \mathbb{N}$, for $x \in (-ct, 0]$
\begin{align}
P_{2k}^-\lbrace M(t) =0,\ \mathcal{T}(t) \in \dif x \rbrace &= P_{2k}^+\lbrace M(t) =\mathcal{T}(t) \in -\dif x \rbrace \label{m0t-p_+}\\
&=P_{2k}^-\lbrace \mathcal{T}(t) \in \dif x \rbrace -P_{2k}^+\lbrace \mathcal{T}(t) \in \dif x \rbrace\label{m0t-p}
\end{align}
and, for $\beta \in (0, ct), \ x\in (2\beta-ct, \beta)$
\begin{equation}\label{met-p}
P_{2k}^-\lbrace M(t) \in \dif \beta,\ \mathcal{T}(t) \in \dif x \rbrace= P_{2k}^+\lbrace M(t) \in \dif \beta,\ \mathcal{T}(t) \in \dif x \rbrace .
\end{equation}
\end{thm}

\begin{proof}
For $k=1$, the reader can prove the theorem by means of (\ref{met+1}) and by proceeding as shown below.

We begin by showing formula (\ref{m0t-p_+}). Let natural $k \ge 2$ and $x \in (-ct, 0]$. At time $T_1 = t_1$ the particle is at level $-ct_1$ and:

- Case 1. It has time to overcome level $0$  and to be in $x$ at time $t$, i.e. if $c(t-t_1)> ct_1 +|x|=ct_1-x \iff t_1 < \frac{ct+x}{2c} \in \bigl(0, \frac{t}{2}\bigr)$.

- Case 2. It has no time to overcome level $0$, but it must be in $x$ at time $t$, i.e. if $c(t-t_1)> ct_1 +x \iff t_1 < \frac{ct-x}{2c} \in \bigl(\frac{t}{2}, t\bigr)$.

In light of these two cases, we proceed as follows
\begin{align*}
P_{2k}^-\lbrace M(t) =0,\ &\mathcal{T}(t) \in \dif x \rbrace \\
&= \int_0^{ \frac{ct+x}{2c}} P_{2k-1}^+\lbrace M(t-t_1) \le ct_1,\ \mathcal{T}(t-t_1) \in \dif x+ct_1 \rbrace P\lbrace T_1 \in \dif t_1 \ |\ N(t) = 2k\rbrace \\
&\ \ \ + \int_{ \frac{ct+x}{2c}}^{\frac{ct-x}{2c}}  P_{2k-1}^+\lbrace \mathcal{T}(t-t_1) \in \dif x+ct_1 \rbrace P\lbrace T_1 \in \dif t_1 \ |\ N(t) = 2k\rbrace\\
& = \int_0^{ \frac{ct-x}{2c}}   P_{2k-1}^+\lbrace \mathcal{T}(t-t_1) \in \dif x+ct_1 \rbrace P\lbrace T_1 \in \dif t_1 \ |\ N(t) = 2k\rbrace\\
&\ \ \ - \int_0^{ \frac{ct+x}{2c}} P_{2k-1}^-\lbrace \mathcal{T}(t-t_1) \in ct_1-\dif x \rbrace P\lbrace T_1 \in \dif t_1 \ |\ N(t) = 2k\rbrace \\
&=P_{2k}^-\lbrace \mathcal{T}(t) \in \dif x \rbrace -P_{2k}^+\lbrace \mathcal{T}(t) \in \dif x \rbrace,
\end{align*}
where in the second step we applied the negative reflection principle for the telegraph process (\ref{principioRiflessioneNegativoTelegrafo2}).

We now prove the joint density (\ref{met-p}). Let natural $k \ge 2,\ \beta \in (0, ct)$ and $x\in (2\beta-ct, \beta)$. At time $T_1 = t_1$ the motion is at position $-ct_1$ and it starts moving with speed $+c$. It must have time to reach $\beta$ and to be in $x$ at time $t$, then we need $c(t-t_1)>ct_1+\beta + |\beta -x| = ct_1+2\beta -x \iff t_1<\frac{ct-2\beta+x}{2c} \in (0, t/2)$.
Now, we can write
\begin{align*}
P_{2k}^-&\lbrace M(t) \in \dif \beta,\, \mathcal{T}(t) \in \dif x \rbrace \\
& = \int_0^{\frac{ct-2\beta+x}{2c}} P_{2k-1}^+\lbrace M(t-t_1) \in \dif \beta +ct_1,\,\mathcal{T}(t-t_1) \in \dif x+ct_1 \rbrace P\lbrace T_1 \in \dif t_1 \, |\, N(t) = 2k\rbrace 
\end{align*}
and by keeping in mind (\ref{met+d}) we obtain that it coincides with the joint density (\ref{met+p}) regarding the positively initially oriented motion with an even number of reversals.
\end{proof}

\begin{rem}
We want to underline the following relationships. From (\ref{m0t-p}), we have, for  natural $k$ and $x \in (-ct, 0]$
\begin{equation}\label{m0t-prip}
P_{2k}^-\lbrace M(t) =0,\ \mathcal{T}(t) \le x \rbrace = \int_{-ct}^x P_{2k}^-\lbrace M(t) =0,\ \mathcal{T}(t) \in \dif y \rbrace = P_{2k+1}^+\lbrace \mathcal{T}(t) \in \dif x \rbrace  \frac{2ct}{(2k+1) \dif x}.
\end{equation}
By keeping in mind formulas (\ref{m0t-p_+}) and (\ref{relazonione_met_m0}) we easily check the following
\begin{equation}
\int_{-ct}^0 P_{2k}^-\lbrace M(t) =0,\ \mathcal{T}(t) \in \dif y \rbrace  =\binom{2k}{k}\frac{1}{2^{2k}} =  P\lbrace M(t) =0\ |\ V(0) = -c,\ N(t) = 2k\rbrace ,
\end{equation}
that is the probability mass that the motion never crosses the level $x=0$ when it starts with a negative initial velocity and changes its direction $2k$ times in the time interval $[0,t]$.
\hfill $\diamond$
\end{rem}

By taking into account the properties of the symmetric telegraph process, it is easy to show that, for $v=\pm c, \ \beta \in (0, ct)$ and natural $n$,
\begin{equation}\label{maxMinCondizionati}
P\{\max_{0\le s\le t}\mathcal{T}(s)\in \dif \beta \ |\ V(0) = v,\ N(t) = n\} =P\{\min_{0\le s\le t}\mathcal{T}(s)\in -\dif \beta \ |\ V(0) = -v,\ N(t) = n\}.
\end{equation}
By considering (\ref{maxMinCondizionati}) and (\ref{m0t-p_+}) we obtain that
\begin{equation*}\label{maxMinCondizionati2}
P\{\max_{0\le s\le t}\mathcal{T}(s) = \mathcal{T}(t)\in \dif \beta \, |\, V(0) = c,\, N(t) = 2k\} =P\{\min_{0\le s\le t}\mathcal{T}(s) = 0,\, \mathcal{T}(t)\in \dif \beta \, |\, V(0) = c,\, N(t) = 2k\},
\end{equation*}
with $\beta\in (0,ct)$ and $k \in \mathbb{N}$. This formula describes an important relationship between the paths of the telegraph process. An analogous equation can be derived for the case of a negative initial velocity (we must switch the maximum and the minimum and consider $\beta \in (-ct,0)$ ).

\begin{cor}\label{met-pripartizioni}
Let $\lbrace \mathcal{T}(t) \rbrace_{t\ge0}$ be a symmetric telegraph process. Let $k \in \mathbb{N}$, for $\beta \in (0, ct), \\ x\in (2\beta-ct, \beta]$
\begin{equation}\label{met-pripx}
P_{2k}^-\lbrace M(t) \in \dif \beta,\ \mathcal{T}(t) <x \rbrace= P_{2k}^+\lbrace M(t) \in \dif \beta,\ \mathcal{T}(t) <x \rbrace,
\end{equation}
and, for $x \in (-ct, ct)$
\begin{align}P_{2k}^-\lbrace M(t)& \le \beta,\ \mathcal{T}(t) \in \dif x \rbrace\nonumber\\
&=\begin{cases}	\begin{array}{l l} 
P_{2k}^-\lbrace \mathcal{T}(t) \in \dif x \rbrace  -P_{2k}^-\lbrace \mathcal{T}(t) \in 2\beta-\dif x \rbrace, &\ if \ \max\lbrace0,x\rbrace\le\beta<\frac{ct+x}{2}, \\
P_{2k}^-\lbrace \mathcal{T}(t) \in \dif x \rbrace, &\ if\ \beta\ge\frac{ct+x}{2}. 
\end{array} \end{cases}\label{met-pripb}
\end{align}
\end{cor}

The reader can notice the incredible fact that (\ref{met-pripb}) and (\ref{leggemax-p}) show that a classical reflection principle holds in the case of a motion that starts moving with a negative velocity and changes direction an even number of times.

\begin{proof}
By means of (\ref{met-p}), we have $P_{2k}^-\lbrace M(t) \in \dif \beta,\ \mathcal{T}(t) < x \rbrace = \int_{2\beta- ct}^x P_{2k}^+\lbrace M(t) \in \dif \beta,\ \mathcal{T}(t) \in \dif y \rbrace$ that proves equation $(\ref{met-pripx})$.

For the second case of result (\ref{met-pripb}) there is nothing to prove. We focus on the first one, where $\max\lbrace0,x\rbrace\le\beta<\frac{ct+x}{2}$. We split it into two cases, $x<0$ and $x\ge0$.
\\Let $x<0$, by means of (\ref{met-p}) and (\ref{m0t-p}), we obtain that
\begin{align*}
P_{2k}^-\lbrace M(t) \le \beta,\ \mathcal{T}(t) \in \dif x \rbrace&= \int_0^\beta P_{2k}^-\lbrace M(t) \in \dif w,\ \mathcal{T}(t) \in \dif x \rbrace  +P_{2k}^-\lbrace M(t) =0,\ \mathcal{T}(t) \in \dif x \rbrace \\
& = P_{2k}^+\lbrace M(t) \le \beta,\ \mathcal{T}(t) \in \dif x \rbrace + \Bigl( P_{2k}^-\lbrace \mathcal{T}(t) \in \dif x \rbrace - P_{2k}^+\lbrace \mathcal{T}(t) \in \dif x \rbrace \Bigr)\\
& =- P_{2k}^-\lbrace \mathcal{T}(t) \in 2\beta -\dif x \rbrace + P_{2k}^-\lbrace \mathcal{T}(t) \in \dif x \rbrace,
\end{align*}
where in the last equation we used the negative reflection principle for the telegraph process (\ref{principioRiflessioneNegativoTelegrafo2}).
\\For $x\ge0$, we have $P_{2k}^-\lbrace M(t) \le \beta,\ \mathcal{T}(t) \in \dif x \rbrace=\int_x^\beta P_{2k}^-\lbrace M(t) \in \dif w,\ \mathcal{T}(t) \in \dif x \rbrace$, and the claimed result follows by considering (\ref{met-p}) and (\ref{met+pint}).
\end{proof}

We now consider the case when $N(t) = 2k+1$ for natural $k\ge0$. In this case, the motion moves with positive speed at time $t$ and therefore it can reach the maximum level also at the end of the time interval, like in the case of an initially positively oriented motion that changes direction an even number of times.

\begin{thm}\label{teoremamet-d}
Let $\lbrace \mathcal{T}(t) \rbrace_{t\ge0}$ be a symmetric telegraph process. Let $k \in \mathbb{N}_0$. For $x \in (-ct, 0]$
\begin{align}\label{m0t-d1}
P_{2k+1}^-\lbrace M(t) =0,\ \mathcal{T}(t) \in \dif x \rbrace&=\binom{2k+1}{k} \frac{(ct-x)^{k-1}(ct+x)^{k}}{(2ct)^{2k+1}}\bigl[ct - (2k+1)x \bigr]\dif x \\
&= P_{2k+1}^-\lbrace \mathcal{T}(t) \in \dif x \rbrace - \frac{(2k+1)!}{(k+1)!(k-1)!} \frac{(ct-x)^{k-1}(ct+x)^{k+1}}{(2ct)^{2k+1}}\dif x. \label{m0t-d}
\end{align}
For $\beta \in (0, ct)$ and $x\in (2\beta-ct, \beta)$ 
\begin{align} P_{2k+1}^-\lbrace M(t) \in \dif \beta,\ \mathcal{T}(t) \in \dif x \rbrace &= \frac{(2k+1)!}{k!(k-1)!} \frac{\bigl[ct-(2\beta-x)\bigr]^{k}\bigl[ct+(2\beta-x)\bigr]^{k-1}}{2^{2k}(ct)^{2k+1}}\dif \beta \dif x\label{met-d}\\
&\ \ \ -  \frac{(2k+1)!}{(k+1)!(k-2)!} \frac{\bigl[ct-(2\beta-x)\bigr]^{k+1}\bigl[ct+(2\beta-x)\bigr]^{k-2}}{2^{2k}(ct)^{2k+1}}\dif \beta \dif x. \nonumber
\end{align}
For $\beta \in (0, ct)$
\begin{align}
 P_{2k+1}^-\lbrace M(t) =\mathcal{T}(t) \in \dif \beta \rbrace& = \binom{2k+1}{k}\frac{(ct-\beta)^{k}(ct+\beta)^{k-1}}{(2ct)^{2k+1}}\bigl[ ct+(2k+1)\beta \bigr] \dif \beta \label{m=t-d} \\
&= P_{2k+1}^-\lbrace \mathcal{T}(t) \in \dif x \rbrace - \frac{(2k+1)!}{(k+1)!(k-1)!} \frac{(ct-x)^{k+1}(ct+x)^{k-1}}{(2ct)^{2k+1}}\dif x.   \nonumber
\end{align}
\end{thm}

\begin{proof}
First of all we prove the theorem for the particular case of $k=0$, i.e. \\$N(t) = 1$. In this case, the joint density (\ref{met-d}) is identically null. On the other hand, we have that the maximum is $0$ if $\mathcal{T}_1^-(t) \le 0$ and therefore equation (\ref{m0t-d}) reduces to the first term. Finally, if the position of the motion at time $t$ is greater than $0$, then the maximum coincides with $\mathcal{T}(t)$ and equation (\ref{m=t-d}) reads $P_1^-\lbrace M(t) = \mathcal{T}(t) \in \dif \beta \rbrace = P_1^-\lbrace M(t) \in \dif \beta \rbrace =P_1^- \lbrace \mathcal{T}(t)  \in \dif \beta \rbrace =\dif \beta/(2ct)$.

For $k\in \mathbb{N}$ the proof of (\ref{m0t-d1}) works like the proof of (\ref{m0t-p}) in Theorem \ref{teoremamet-p} and therefore it is omitted.

The joint density (\ref{met-d}) follows by applying the same arguments we used in the proof of density (\ref{met-p}).

We conclude by showing formula (\ref{m=t-d}). For $\beta \in (0, ct)$,
$$  P_{2k+1}^-\lbrace M(t) =\mathcal{T}(t) \in \dif \beta \rbrace =\int_0^{\frac{ct-\beta}{2c}}  P_{2k}^+\lbrace M(t-t_1) =\mathcal{T}(t-t_1) \in \dif \beta +ct_1\rbrace P\lbrace T_1 \in \dif t_1 \, |\, N(t) = 2k+1\rbrace,$$
where the integration set is obtained by considering that the first probability in the integral is $0$ if $c(t-t_1) < \beta+ ct_1$. The claimed result follows by using (\ref{m=t+p}).
\end{proof}

\begin{rem}
Let integer $k\ge0$ and $x \in (-ct, 0]$, by means of (\ref{m0t-d}) we obtain
\begin{equation*}\label{relazioneParticolarem=0ripx}
  P_{2k+1}^-\lbrace M(t) = 0 ,\ \mathcal{T}(t) \le x  \rbrace=   \frac{2ctP_{2k+2}^+\lbrace\mathcal{T}(t)  \in \dif x \rbrace}{(2k+2) \dif x} =\frac{(2k+1)!}{(k+1)!k!}\frac{(c^2t^2-x^2)^k(ct+x)}{(2ct)^{2k+1}},
\end{equation*}
which resembles the relationship (\ref{m0t-prip}) in the case of an even number of changes of direction. Furthermore, we immediately verify that
\begin{equation}\label{relazioneM0tEm=0}
P_{2k+1}^-\lbrace M(t) = 0 ,\ \mathcal{T}(t) \le 0  \rbrace  =  \binom{2k+1}{k}\frac{1}{2^{2k+1}} = P_{2k+1}^-\lbrace M(t) =0\rbrace,
\end{equation}
that is the probability that an initially negatively oriented motion, when $2k+1$ changes of direction occur before time $t$, does not cross the starting level $x=0$ in the time interval $[0,t]$. Incredibly, this probability mass is independent of both $t$ and $c$.
\hfill $\diamond$
\end{rem}

\begin{rem}
We point out that taking into account formulas (\ref{m0t-d1}) and (\ref{m=t-d}), for $\beta\in (0,ct)$
\begin{equation}\label{relazione_m=t_m0t-d}
 P_{2k+1}^-\lbrace M(t) =\mathcal{T}(t) \in \dif \beta \rbrace =  P_{2k+1}^-\lbrace M(t) =0,\ \mathcal{T}(t) \in -\dif \beta \rbrace.
\end{equation}
Therefore, by considering equation (\ref{relazioneM0tEm=0}), we obtain that
\begin{equation}
P_{2k+1}^-\lbrace M(t) =\mathcal{T}(t) \rbrace = \binom{2k+1}{k}\frac{1}{2^{2k+1}} =  P_{2k+1}^-\lbrace M(t) =0 \rbrace.
\end{equation}
This equality suggests an interesting connection between the samples of the sets in the first and the third members, concerning the initially negatively oriented motion when an odd number of reversals occur.
\hfill $\diamond$
\end{rem}

\begin{cor}\label{met-dripartizioni}
Let $\lbrace \mathcal{T}(t) \rbrace_{t\ge0}$ be a symmetric telegraph process. Let $k \in \mathbb{N}_0$, for $\beta \in (0, ct), \ x\in (2\beta-ct, \beta]$
\begin{equation}\label{met-dripx}
P_{2k+1}^-\lbrace M(t) \in \dif \beta,\ \mathcal{T}(t) <x \rbrace\ = \frac{(2k+1)!}{(k+1)!(k-1)!} \frac{\bigl[ct-(2\beta -x) \bigr]^{k+1} \bigl[ct +(2\beta-x) \bigr]^{k-1}}{2^{2k}(ct)^{2k+1}}\dif \beta
\end{equation}
and, for $x \in (-ct, ct)$
\begin{align}
P&_{2k+1}^-\lbrace M(t) \le \beta,\ \mathcal{T}(t) \in \dif x \rbrace \label{met-dripb}\\
&=\bigg \lbrace\begin{array}{l l} 
	P_{2k+1}^-\lbrace \mathcal{T}(t) \in \dif x \rbrace  -  \frac{(2k+1)!}{(k+1)!(k-1)!}\frac{\bigl[ct-(2\beta -x) \bigr]^{k+1}\bigl[ct+(2\beta -x) \bigr]^{k-1}}{(2ct)^{2k+1}}\dif x&\ if \ \max\lbrace0,x\rbrace\le\beta<\frac{ct+x}{2} \\
	P_{2k+1}^-\lbrace \mathcal{T}(t) \in \dif x \rbrace &\ if\ \beta\ge\frac{ct+x}{2} \\
	\end{array} \nonumber
\end{align}
\end{cor}

\begin{proof}
The proof of probability (\ref{met-dripx}) follows by using formula (\ref{met-d}) and by proceeding as we showed in the proof of Corollary \ref{corollariomet+pripx}.

For the second case of result (\ref{met-dripb}) there is nothing to prove. We focus on the first one, where $\max\lbrace0,x\rbrace\le\beta<\frac{ct+x}{2} $. We consider two cases, $x<0$ and $x\ge0$.
\\Let $x<0$, by using (\ref{met-d}), some calculation yield
\begin{align}
 \int_0^\beta P_{2k+1}^-\lbrace &M(t) \in \dif w,\ \mathcal{T}(t) \in \dif x \rbrace\label{met-dripb_integrale1}\\
&=\frac{(2k+1)!}{(k+1)!(k-1)!}\Biggl[ \frac{(ct-x)^{k-1} (ct+x)^{k+1}}{(2ct)^{2k+1}}- \frac{\bigl[ct-(2\beta-x) \bigr]^{k+1} \bigl[ct+(2\beta-x) \bigr]^{k-1}}{(2ct)^{2k+1}} \Biggr]\dif x.\nonumber
\end{align}
Now, we have
$$P_{2k+1}^-\lbrace M(t) \le \beta,\ \mathcal{T}(t) \in \dif x \rbrace = P_{2k+1}^-\lbrace M(t) =0,\ \mathcal{T}(t) \in \dif x \rbrace  +\int_0^\beta P_{2k+1}^-\lbrace M(t) \in \dif w,\ \mathcal{T}(t) \in \dif x \rbrace$$
and the first equation in (\ref{met-dripb}) follows by using (\ref{m0t-d}) and (\ref{met-dripb_integrale1}).
\\For $x\ge0$, we have
$$P_{2k+1}^-\lbrace M(t) \le \beta,\ \mathcal{T}(t) \in \dif x \rbrace =P_{2k+1}^-\lbrace M(t) =\mathcal{T}(t) \in \dif x \rbrace +  \int_x^\beta P_{2k+1}^-\lbrace M(t) \in \dif w,\ \mathcal{T}(t) \in \dif x \rbrace.$$
The second term follows by means of (\ref{met-d}). By also considering (\ref{m=t-d}) we obtain (\ref{met-dripb}).
\end{proof}

By applying (\ref{m=t-d}) and (\ref{met-dripx}) it is easy to check that
\begin{align*}
P_{2k+1}^-\lbrace M(t)& \in \dif \beta,\ \mathcal{T}(t) <\beta \rbrace + P_{2k+1}^-\lbrace M(t) = \mathcal{T}(t) \in \dif \beta \rbrace \\
& =\binom{2k+1}{k} \frac{(ct-\beta)^{k} (ct+\beta)^{k-1}}{(2ct)^{2k+1}}\bigl[(2k+1)ct+\beta \bigr]\dif \beta =  P_{2k+1}^-\lbrace M(t) \in \dif \beta \rbrace
\end{align*}
which coincides with the result of Theorem 4.2 of \cite{CO2020}.
\\

We conclude this section by showing the distribution of the telegraph process and its maximum conditioned on a negative starting velocity.

\begin{thm}
Let $\{\mathcal{T}(t)\}_{t\ge0}$ be a symmetric telegraph process. Let $x\in (-ct,0]$
\begin{align}
P\{M(t) &=0,\ \mathcal{T}(t)\in \dif x\ |\ V(0)=-c\}\label{m0t-}\\
&=e^{-\lambda t}\Biggl[ \frac{-\lambda x}{c(ct-x)}\,I_0\Bigl(\frac{\lambda}{c}\sqrt{c^2t^2-x^2} \Bigr) + \frac{1}{\sqrt{c^2t^2-x^2}}\biggl(\frac{ct+x}{ct-x}-\frac{\lambda x}{c} \biggr)I_1\Bigl(\frac{\lambda}{c}\sqrt{c^2t^2-x^2} \Bigr)\Biggr]\dif x,\nonumber
\end{align}
and
\begin{equation}\label{m0t=-ct-}
P\{M(t) =0,\ \mathcal{T}(t)=-ct\ |\ V(0)=-c\}= e^{-\lambda t}.
\end{equation}
For $\beta \in (0,ct), \ x\in(2\beta-ct,\beta)$
\begin{align}
P\{M(t) \in \dif \beta&,\ \mathcal{T}(t)\in \dif x\ |\ V(0)=-c\}\label{met-}\\
& =\ \frac{\lambda^2\,e^{-\lambda t}}{2c^2}\, \Biggl[\ \sum_{j=0}^1\biggl(\sqrt{\frac{ct-(2\beta-x)}{ct+(2\beta-x)}}\biggr)^j\, I_j\Bigl(\frac{\lambda}{c}\sqrt{c^2t^2-(2\beta-x)^2}\Bigr)\nonumber\\
&\ \ \ - \sum_{j=2}^3\biggl(\sqrt{\frac{ct-(2\beta-x)}{ct+(2\beta-x)}}\biggr)^j\, I_j\Bigl(\frac{\lambda}{c}\sqrt{c^2t^2-(2\beta-x)^2}\Bigr)\Biggr] \dif \beta \dif x,\nonumber
\end{align}
and
\begin{equation}\label{m=t-}
P\{M(t) = \mathcal{T}(t)\in \dif \beta\ |\ V(0)=-c\} = \frac{e^{-\lambda t}}{ct+\beta}\Biggl[ \frac{\lambda \beta}{c}\,I_0\Bigl(\frac{\lambda}{c}\sqrt{c^2t^2-\beta^2}\Bigr)+\sqrt{\frac{ct-\beta}{ct+\beta}}\,I_1\Bigl(\frac{\lambda}{c}\sqrt{c^2t^2-\beta^2}\Bigr) \Biggr]\dif \beta.
\end{equation}
\end{thm}

\begin{proof}
Let $x\in (-ct,0]$. By means of (\ref{m0t-p}) we obtain that
\begin{equation}\label{m0t-Pari}
P\{M(t) =0,\ \mathcal{T}(t)\in \dif x,\ N(t)\ \text{even}\ |\ V(0)=-c\}=\frac{-\lambda x\,e^{-\lambda t}}{\sqrt{c^2t^2-x^2}}\,I_1\Bigl(\frac{\lambda}{c}\sqrt{c^2t^2-x^2}\Bigr) \dif x .
\end{equation}
By using (\ref{m0t-d}) and the recurrence relationship for the Bessel function (\ref{relazioneBessel}) with $n=1$, we have
\begin{align}
P\{M(t) =0,&\ \mathcal{T}(t)\in \dif x,\  N(t)\ \text{odd}\ |\ V(0)=-c\}\nonumber\\
&= \frac{e^{-\lambda t}}{ct-x}\Biggl[ \frac{-\lambda x}{c}\,I_0\Bigl(\frac{\lambda}{c}\sqrt{c^2t^2-x^2}\Bigr)+\sqrt{\frac{ct+x}{ct-x}}\,I_1\Bigl(\frac{\lambda}{c}\sqrt{c^2t^2-x^2}\Bigr) \Biggr]\dif x.\label{m0t-Dispari2}
\end{align}
\\By summing up (\ref{m0t-Pari}) and (\ref{m0t-Dispari2}) we obtain the claimed result (\ref{m0t-}).

The probability mass (\ref{m0t=-ct-}) coincides with the probability that no Poisson events occur in $[0,t]$.

To prove the joint density (\ref{met-}) we consider the following distributions. By means of (\ref{met-p}) and (\ref{met+pcomoda}) we obtain that
\begin{align}
P\{M&(t) \in \dif \beta,\ \mathcal{T}(t)\in \dif x,\ N(t)\ \text{even}\ |\ V(0)=-c\}\label{met-Pari}\\
& = \frac{\lambda^2\,e^{-\lambda t}}{2c^2}\,\Biggl[\ I_0\Bigl(\frac{\lambda}{c}\sqrt{c^2t^2-(2\beta-x)^2}\Bigr) - \frac{ct-(2\beta-x)}{ct+(2\beta-x)}\,I_2\Bigl(\frac{\lambda}{c}\sqrt{c^2t^2-(2\beta-x)^2}\Bigr)\Biggr] \dif \beta \dif x.\nonumber
\end{align}
and by means of (\ref{met-d}) we have that
\begin{align}
P\{M(t) \in& \dif \beta,\ \mathcal{T}(t)\in \dif x,\ N(t)\ \text{odd}\ |\ V(0)=-c\}\label{met-Dispari}\\
&= \frac{\lambda^2\,e^{-\lambda t}}{2c^2}\ \sum_{j=0}^1(-1)^j\biggl(\sqrt{\frac{ct-(2\beta-x)}{ct+(2\beta-x)}}\biggr)^{2j+1}\, I_{2j+1}\Bigl(\frac{\lambda}{c}\sqrt{c^2t^2-(2\beta-x)^2}\Bigr) \dif \beta \dif x.\nonumber
\end{align}
\\By summing up (\ref{met-Pari}) and (\ref{met-Dispari}) we obtain the distribution (\ref{met-}).

Finally, to prove (\ref{m=t-}) we just need to consider equation (\ref{relazione_m=t_m0t-d}) and probability (\ref{m0t-Dispari2}).
\end{proof}

\begin{rem}
From the distributions we displayed in this section, the interested reader can obtain some interesting conditional and unconditional probabilities. For instance it is easy to study the distribution of the maximum conditioned on the position of the process at the end of the time interval, that is the maximum of a telegraph bridge. It is also easy to compute the probability law of the telegraph process in presence of an absorbing barrier and the distribution of some kind of the telegraph meander.\hfill $\diamond$
\end{rem}


\section{On the random times of the symmetric telegraph process}

In this section we focus on the conditional and unconditional distributions of the first passage time and the returning time to the origin of the telegraph process.

\begin{rem}
We show the following interesting relationship between the time derivative of the cumulative distribution functions of the maximum of the telegraph process and its density, when we know the starting speed $V(0)$ and the number of changes of direction $N(t)$. For $0<\frac{\beta}{c}<t, \ n\in \mathbb{N}_0$
\begin{equation}\label{derivatamax}
-\frac{\partial }{\partial t}P^\pm_n\{ M(t) \le\beta \}= \frac{\beta}{t\,\dif \beta}P_n^\pm\{ M(t)\in \dif \beta\}.
\end{equation}

In the case of $V(0) = +c$, from Corollary 5.1 of \cite{CO2020}, we have that, for natural $k\ge 0,\ \beta \in [0,ct]$
\begin{equation}\label{rip+paridispari}
P^+_{2k+1}\{ M(t) \le\beta \} =P^+_{2k+2}\{ M(t) \le\beta \} = \frac{\beta}{ct}\sum_{j=0}^k \binom{2j}{j}\frac{(c^2t^2-\beta^2)^j}{(2ct)^{2j}}
\end{equation}
and we point out that (\ref{leggemax+d2}) has the same cyclic behavior displayed by equation (\ref{rip+paridispari}).
\\Now, by suitably adapting (\ref{leggemax+d2}) we have
\begin{align}
 -&\frac{\partial }{\partial t}P^+_{2k+1}\{ M(t) \le\beta \} = -\frac{\partial }{\partial t}P^+_{2k+2}\{ M(t) \le\beta \} =-\frac{\partial }{\partial t}\int_0^\beta \frac{2(2k+1)!}{k!^2}\frac{(c^2t^2-w^2)^k}{(2ct)^{2k+1}}\dif w\nonumber\\
&= - \frac{2(2k+1)!}{k!^2\: 2^{2k+1}}\,\frac{\partial }{\partial t}\int^{\frac{\beta}{ct}}_0 (1-y^2)^k\dif y  = \frac{2(2k+1)!}{k!^2\: 2^{2k+1}} \, \frac{\beta}{ct^2}\Bigl(1-\frac{\beta^2}{c^2t^2}\Bigr)^k=\frac{\beta}{t\,\dif \beta}P_{2k+1}^+\{ M(t)\in \dif \beta\}\label{dermax+}
\end{align}
which concludes the proof when $V(0) = c$.

In the case of a negative starting speed, $V(0) = -c$, by means of (\ref{rip+paridispari}) and Corollary 5.2 of \cite{CO2020} we can write the following relationship for the cumulative distribution function, for $\beta\in [0,ct],\ k \in \mathbb{N}$
\begin{equation}\label{massimo-prip}
P_{2k}^-\lbrace M(t) \le \beta\rbrace\ =\ P_{2k}^+\lbrace M(t) \le \beta\rbrace +  \binom{2k}{k} \frac{(c^2t^2-\beta^2)^k }{(2ct)^{2k}}.
\end{equation}
Now, in light of result (\ref{dermax+}) and by means of (\ref{massimo-prip}), we can write
\begin{align}
-\frac{\partial }{\partial t}P^-_{2k}\{ M(t) \le\beta \} &=\frac{\beta}{t\,\dif \beta}P^+_{2k}\{ M(t) \in \dif \beta \}  - \binom{2k}{k}\frac{(2k)(c^2t^2-\beta^2)^{k-1}}{(2ct)^{2k}\,t}\beta^2\nonumber\\
&=  \frac{2(2k)!}{k!(k-1)!} \frac{(c^2t^2-\beta^2)^{k-1}(ct-\beta)}{2^{2k}(ct)^{2k}}\frac{\beta}{t} = \frac{\beta}{t\,\dif \beta}P^-_{2k}\{ M(t) \in \dif \beta \}. \label{dermax-p}
\end{align} 
For the odd case, from \cite{CO2020} we can write
\begin{equation}\label{massimo-drip}
P_{2k+1}^-\lbrace M(t) \le \beta\rbrace\ =\ \frac{2k+1}{2k+2} P_{2k}^-\lbrace M(t) \le \beta\rbrace + \frac{1}{2k+2} P_{2k+1}^+\lbrace M(t) \le \beta\rbrace 
\end{equation}
and (\ref{derivatamax}) follows by using results (\ref{dermax+}) and (\ref{dermax-p}).\hfill$\diamond$
\end{rem}

It is interesting to observe that relationship (\ref{derivatamax}) is typical of self-similar processes. However, the telegraph process is not self-similar in the classical sense, but it enjoys a kind of quasi-similarity property with $H=1$. In fact, the time-scaled process $\{\mathcal{T}(at)\}_{t\ge0}$, with $a>0$, can be explicitly written as
$$ \mathcal{T}(at)=V(0)\int_0^{at}(-1)^{N(s)}\dif s=aV(0)\int_0^t(-1)^{N(as)}\dif s, $$
that is a scaled telegraph motion, but with rate $a\lambda>0$.

\subsection{First passage time: motion starting with positive velocity}

Let $F_\beta = \inf\{s\ge 0\: : \: \mathcal{T}(s) \ge \beta\}$ the first passage time of the symmetric telegraph process across level $\beta$. Note that for $\beta>0$ and natural $n$
\begin{equation}\label{relazioneTempiBetaV}
P\{F_\beta\in \dif s\ |\ V(0)=c,\ N(t) = n\} =P\{F_{-\beta}\in \dif s\ |\ V(0)=-c,\ N(t) = n\}
\end{equation}
and therefore we limit ourselves to study the first passage time across a level $\beta>0$ for both the initial possible velocities.

It is clear that the telegraph process starting with positive velocity can hit level $\beta>0$ at time $t\ge \frac{\beta}{c}$ and only if $N(t)$ is even. In particular, if $N(t)=0$, the first passage through $\beta$ occurs at time $\frac{\beta}{c}$ with probability one. In the following theorem we study the distribution of the first passage time, by assuming the number of changes of direction. Also in this case, $T_\beta\ge\beta/c\ a.s.$

\begin{thm}\label{teorematppN+}
Let $F_\beta$ be the first passage time of a symmetric telegraph process across level $\beta>0$. Let $0<\frac{\beta}{c}<s\le t$. For $k \in \mathbb{N}$
\begin{equation}\label{tppN+}
P\{F_\beta \in \dif s\ |\ V(0)=c,\ N(t) = 2k\} = \frac{(2k)!}{t^{2k}}\beta \sum_{j=1}^k \frac{(t-s)^{2k-2j}}{j!(j-1)!(2k-2j)!}\frac{(c^2s^2-\beta^2)^{j-1}}{(2c)^{2j-1}}\dif s
\end{equation}
and
\begin{equation*}\label{tppN+sing}
P\Big\{F_\beta = \frac{\beta}{c}\ \Big|\ V(0)=c,\ N(t) = 2k\Big\} = \Bigl(1-\frac{\beta}{ct} \Bigr)^{2k}.
\end{equation*}
For $k \in \mathbb{N}_0$
\begin{equation}\label{tppNdispari+}
P\{F_\beta \in \dif s\ |\ V(0)=c,\ N(t) = 2k+1\} = \frac{(2k+1)!}{t^{2k+1}}\beta \sum_{j=1}^k \frac{(t-s)^{2k+1-2j}}{j!(j-1)!(2k+1-2j)!}\frac{(c^2s^2-\beta^2)^{j-1}}{(2c)^{2j-1}}\dif s
\end{equation}
and
\begin{equation*}\label{tppNdispari+sing}
P\Big\{F_\beta = \frac{\beta}{c}\ \Big|\ V(0)=c,\ N(t) = 2k+1\Big\} = \Bigl(1-\frac{\beta}{ct} \Bigr)^{2k+1}.
\end{equation*}
\end{thm}

We point out that the distribution for $s>t$ is not displayed. It is interesting to observe that for $\beta \downarrow 0$, $F_\beta \longrightarrow 0\,a.s.$, conditionally on $V(0) =c$ and the number of switches. Note that if $N(t) = 1$, the absolutely continuous component (\ref{tppNdispari+}) is null.

\begin{proof}
Let $k \in \mathbb{N},\ 0<\frac{\beta}{c}<s\le t$. By considering that $P\big\{F_\beta = \frac{\beta}{c}\ \big|\ V(0)=c,\ N(t) = 2k\big\} = P\big\{N\bigl(\frac{\beta}{c}\bigr) = 0\ \big|\ V(0)=c,\ N(t) = 2k\big \}$ we obtain result (\ref{tppN+sing}).

Throughout this proof we modify notation (\ref{notazionemassimo&telegrafo}) by considering the Poisson process at time $s>0$, $N(s)$, instead of at time $t\ge s$, therefore
\begin{equation}\label{notazionemassimo&telegrafo2}
P_{n}^\pm\{\ \cdot\ \} = P\{\ \cdot\ |\ V(0) = \pm c,\ N(s) = n\}.
\end{equation} 
To prove (\ref{tppN+}) we start by studying the following distribution of the maximum. For integer $k\ge1,\ 0<\frac{\beta}{c}<s\le t$,
\begin{align}
P\{ M(s) \le\beta\ |\ V(0) = c,\ N(t) = 2k\}& = \sum_{j=0}^{2k} P\{ M(s) \le\beta,\ N(s) = j\ |\ V(0)=c,\ N(t) = 2k\}\nonumber\\
& =\sum_{j=0}^{2k} P_j^+\{M(s) \le\beta \}\,\frac{P\{N(s) = j\}\,P\{N(t-s)=2k-j\}}{P\{N(t) = 2k\}} \nonumber\\
&=\frac{1}{t^{2k}}\sum_{j=0}^{2k} P_j^+\{M(s) \le\beta \}\binom{2k}{j}s^j(t-s)^{2k-j}\label{max+st},
\end{align}
where in the second equality we used that $M(s)$ and $N(t)$ are conditionally independent with respect to $N(s)=j$. We can now can study the distribution of the first passage time.
\begin{align}
P\{F_\beta \in \dif s\ |\ V(0)=&c,\ N(t) = 2k\}/\dif s =  -\frac{\partial}{\partial s} P\{ M(s) \le\beta\ |\ V(0) = c,\ N(t) = 2k\} \nonumber\\
&\begin{array}{l}\label{tpp+_lunga}
\displaystyle
= \frac{(2k)!}{t^{2k}}\: \sum_{j=0}^{2k}\frac{1}{j!(2k-j)!}\Biggl[\frac{\beta}{s\, \dif \beta}P_j^+\{M(s)\in \dif \beta\}\,s^j(t-s)^{2k-j}\ +\\
\ \ \ - P_j^+\{M(s)\le \beta\}\biggl(js^{j-1}(t-s)^{2k-j}-(2k-j)s^j(t-s)^{2k-j-1} \biggr)\Biggr] ,
\end{array}
\end{align}
where in the last equality we used (\ref{max+st}) and (\ref{derivatamax}).
\\We can write the first term of (\ref{tpp+_lunga}) as follows
\begin{align}
\frac{\beta}{s\, \dif \beta}\sum_{j=1}^{2k}& \frac{s^j(t-s)^{2k-j}}{j!(2k-j)!} P_j^+\{M(s)\in \dif \beta\} = \frac{\beta}{\dif \beta}\sum_{l=0}^{2k-1} \frac{s^{l}(t-s)^{2k-1-l}}{(l+1)!(2k-1-l)!} P_{l+1}^+\{M(s)\in \dif \beta\} \label{tpp+_addendo1} \\
& = \frac{\beta}{cs}\sum_{h=0}^{k-1} \frac{(t-s)^{2k-1-2h }}{(2k-1-2h )!} \frac{(c^2s^2-\beta^2)^h}{h!^2(2c)^{2h }} \:+\: \frac{\beta}{\dif \beta}\,\sum_{h=0}^{k-1} \frac{s^{2h +1}(t-s)^{2k-2-2h }}{(2h +2)!(2k-2-2h )!} P_{2h +2}^+\{M(s)\in \dif \beta\},\nonumber
\end{align}
where the two sums concern the indexes $l$ even and odd respectively. We simplify the second term of (\ref{tpp+_lunga}) as follows (note that, from (\ref{rip+paridispari}), if $j$ odd, $P_{j}^+\{M(s)\le \beta\} =P_{j+1}^+\{M(s)\le \beta\}$ )
\begin{align}
\sum_{j=0}^{2k} \frac{P_j^+\{M(s)\le \beta\}}{j!(2k-j)!}&\biggl[(2k-j)s^j(t-s)^{2k-j-1}-js^{j-1}(t-s)^{2k-j} \biggr] = \nonumber\\
& = \sum_{j=0}^{2k-1}  P_j^+\{M(s)\le \beta\} \frac{s^j(t-s)^{2k-1-j}}{j!(2k-1-j)!} - \sum_{l=0}^{2k-1}  P_{l+1}^+\{M(s)\le \beta\} \frac{s^l(t-s)^{2k-1-l}}{l!(2k-1-l)!} \nonumber\\
 &= \sum_{j=0}^{k-1}\frac{s^{2j}(t-s)^{2k-1-2j}}{(2j)!(2k-1-2j)!} \Biggl( P_{2j}^+\{M(s)\le \beta\} -P_{2j+1}^+\{M(s)\le \beta\} \Biggr) \nonumber\\
& =  - \frac{\beta}{cs}\sum_{j=0}^{k-1} \frac{(t-s)^{2k-1-2j}}{j!^2(2k-1-2j)!} \frac{(c^2s^2-\beta^2)^j}{(2c)^{2j}},  \label{tpp+_addendo2}
\end{align} 
where in the last equation we used (\ref{rip+paridispari}) and compute some trivial calculations. Finally, by applying (\ref{tpp+_addendo1}) and (\ref{tpp+_addendo2}) in formula (\ref{tpp+_lunga}) we obtain the claimed result (\ref{tppN+}).

The proof for the odd case work in the same way and therefore it is omitted.
\end{proof}

It is important to notice that, for natural $k$ and $0<\frac{\beta}{c}<t$
\begin{align}
P \{F_\beta \in \dif t\ |\ V(0)=c,\ N(t) = 2k\} &= \frac{2(2k-1)!}{(k-1)!^2}\frac{(c^2t^2-\beta^2)^{k-1}}{(2ct)^{2k-1}} \frac{\beta}{t} \dif t\label{tppN+t}\\
& =  \frac{\beta\, \dif t}{t \, \dif \beta}  P\{M(t) \in \dif \beta\ |\ V(0)=c,\ N(t) = 2k\}.\label{tppN+t_relazione}
\end{align}
Formula (\ref{tppN+t_relazione}) connect the density of the first passage time across level $\beta$ with the density of the maximum in a fascinating way which resembles the relationship between the two distributions in the case of Brownian motion.

\begin{thm}
Let $F_\beta$ be the first passage time of a symmetric telegraph process across level $\beta>0$. For $0<\frac{\beta}{c}< t$,
\begin{equation}\label{tpp+}
P\{F_\beta \in \dif t\ |\ V(0)=c\} = \frac{\lambda \beta\,e^{-\lambda t}}{\sqrt{c^2t^2-\beta^2}} I_1\Bigl( \frac{\lambda}{c}\sqrt{c^2t^2-\beta^2}\Bigr) \dif t
\end{equation}
and
\begin{equation}\label{tpp+sing}
P\Big\{F_\beta = \frac{\beta}{c}\ \Big|\ V(0)=c\Big\} =e^{-\lambda \frac{\beta}{c}}.
\end{equation}
\end{thm}

\begin{proof}
By means of (\ref{tppN+t_relazione}) we have
\begin{align*}
P\{F_\beta \in \dif t\ |\ V(0)=c\}& =\frac{\beta\, \dif t}{t \, \dif \beta} \sum_{k=1}^{\infty} P\{M(t) \in \dif \beta,\ N(t) = 2k\ |\ V(0) = c\} \\
&=\frac{\beta\, \dif t}{t \, \dif \beta} P\{M(t)\in \dif \beta, N(t) \ \text{even}\ |\ V(0) = c \},
\end{align*}
and probability (\ref{tpp+}) easily follows by considering that
\begin{equation*}
 P\{M(t)\in \dif \beta,\ N(t) \ \text{even}\ |\ V(0) = c \}=\frac{ \lambda t\,e^{-\lambda t}}{\sqrt{c^2t^2 -\beta^2}} I_1\Bigl( \frac{\lambda }{c} \sqrt{c^2t^2 -\beta^2} \Bigr)\dif \beta,
\end{equation*}
that is formula (3.23) of \cite{CO2020}.
\end{proof}

Clearly, for $\beta \downarrow 0$, $F_\beta \longrightarrow 0\,a.s.$, if we know that $V(0) =c$. Moreover, distribution (\ref{tpp+}) converges to the distribution of the first passage time of Brownian motion under Kac's conditions, $\lambda,c \longrightarrow \infty$ such that $\frac{\lambda}{c^2}\longrightarrow 1$. The reader can prove this assertion by considering the asymptotic estimate of the Bessel functions, when $x\longrightarrow \infty$, $I_r(x)\sim \frac{e^x}{\sqrt{2\pi x}}$ for $r=0,1$.

\subsection{First passage time: motion starting with negative velocity}

In the case of the telegraph process starting with negative speed, the first passage time across a level $\beta>0$ can occur at time $t>\frac{\beta}{c}$ and only if $N(t)$ is odd. In this case the unconditional distribution has no singularity in $\frac{\beta}{c}$ and $T_\beta > \beta/c \ a.s.$ under any condition.

\begin{thm}
Let $F_\beta$ be the first passage time of a symmetric telegraph process across level $\beta>0$. For $0<\frac{\beta}{c}<s\le t$ and $k \in \mathbb{N}_0$
\begin{align}
P\{F_\beta \in \dif s\ |\ &V(0)=-c,\ N(t) = 2k+1\} \label{tppN-}\\
&= \frac{(2k+1)!}{t^{2k+1}} \sum_{j=0}^k \frac{(t-s)^{2k-2j}}{j!(j+1)!(2k-2j)!}\frac{(c^2s^2-\beta^2)^{j-1}(cs-\beta)\bigl[cs+(2j+1)\beta\bigr]}{2^{2j+1}c^{2j}}\dif s\nonumber
\end{align}
and, for $k \in \mathbb{N}$
\begin{align}
P\{F_\beta \in \dif s\ |\ &V(0)=-c,\ N(t) = 2k\} \label{tppNpari-}\\
&= \frac{(2k)!}{t^{2k}} \sum_{j=0}^k \frac{(t-s)^{2k-1-2j}}{j!(j+1)!(2k-1-2j)!}\frac{(c^2s^2-\beta^2)^{j-1}(cs-\beta)\bigl[cs+(2j+1)\beta\bigr]}{2^{2j+1}c^{2j}}\dif s\nonumber
\end{align}
\end{thm}
We point out that the distribution for $s>t$ is not displayed. Note that for $\beta \downarrow 0$, probability (\ref{tppN-}) reduces to $\frac{(2k+1)!}{2t^{2k+1}} \sum_{j=0}^k \frac{(t-s)^{2k-2j}s^{2j}}{2^{2j}j!(j+1)!(2k-2j)!}\dif s$, which does not depend on the velocity of the motion (an equivalent result follows from (\ref{tppNpari-}) ).

\begin{proof}
To prove the theorem we must proceed as we showed in the proof of Theorem \ref{teorematppN+}. In this case, the difference is that, considering notation (\ref{notazionemassimo&telegrafo2}), for integer $k \ge0$, we have
\begin{equation}
P_{2k+1}^-\{M(s)\le\beta\}-P_{2k+2}^-\{M(s)\le\beta\} = -2\beta \binom{2k+1}{k}\frac{(c^2s^2-\beta^2)^k}{(2cs)^{2k+2}}(cs-\beta),
\end{equation}
which follow from (\ref{massimo-prip}) and (\ref{massimo-drip}). Note that in the case of $V(0) = c$ this difference is null.
\end{proof}

It is important to notice that, for $k\in \mathbb{N}_0,\ 0<\frac{\beta}{c}<t$
\begin{equation}\label{tppN-t}
P \{F_\beta \in \dif t\ |\ V(0)=-c,\ N(t) = 2k+1\} = \binom{2k+1}{k}\frac{(c^2t^2-\beta^2)^{k-1}(ct-\beta)\bigl[ct+(2k+1)\beta \bigr]}{c^{2k}(2t)^{2k+1}} \dif t.
\end{equation}

\begin{thm}
Let $F_\beta$ be the first passage time of a symmetric telegraph process across level $\beta>0$. For $0<\frac{\beta}{c}< t$
\begin{equation}\label{tpp-}
P\{F_\beta \in \dif t\ |\ V(0)=-c\}  = \frac{e^{-\lambda t}}{ct+\beta} \Biggr[ \lambda \beta I_0\Bigl( \frac{\lambda}{c}\sqrt{c^2t^2-\beta^2}\Bigr)+ c\sqrt{\frac{ct-\beta}{ct+\beta}} I_1\Bigl( \frac{\lambda}{c}\sqrt{c^2t^2-\beta^2}\Bigr)\Biggr] \dif t.
\end{equation}
\end{thm}

\begin{proof}
$$P\{F_\beta \in \dif t\ |\ V(0)=-c\} = \sum_{k=0}^{\infty} P\{F_\beta \in \dif t\ |\ V(0)=-c,\ N(t) = 2k+1\}P\{N(t) = 2k+1\}$$
and the result (\ref{tpp-}) follows by means of (\ref{tppN-t}) (we suggest to rewrite its last factor considering that $\bigl[ct+(2k+1)\beta \bigr] = (2k+2)\beta +(ct-\beta)$ ).
\end{proof}

Under Kac's conditions, distribution (\ref{tpp-}) converges to the distribution of the first passage time of Brownian motion. Finally, if $\beta\downarrow 0$, density (\ref{tpp-}) reads $e^{-\lambda t}/t\,I_1(\lambda t)$ and as above, this quantity does not depend on $c$.
\\

We want to recall that distributions (\ref{tppN+t}), (\ref{tpp+}), (\ref{tppN-t}) and  (\ref{tpp-}) were first proved in \cite{FK1994} by means of the Darlin-Siegert's relationship. The reader can also find more general expression in Theorem 3.8 of \cite{KR2013}.

\subsection{Return time to the origin}

The results on the first passage time easily yield the distribution of the first returning time to the origin for the telegraph process, $F_0 =\inf\{s>0\ :\ \mathcal{T}(s)=0\}$.
It is clear that the probabilities of the first returning time conditioned on $V(0)=c$ coincide with those conditioned on $V(0)=-c$.

\begin{thm}
Let $F_0$ be the first returning time to the origin of a symmetric telegraph process. Let $v=\pm c$. For $0<s\le t$
\begin{equation}\label{tpp0N1}
P\{F_0\in \dif s\ |\ V(0)=v,\ N(t) = 1\} = \frac{\dif s}{2t},
\end{equation}
 and, for $k \in \mathbb{N}$
\begin{align}\label{tpp0Ndispari}
P\{F_0 \in \dif s\ |\ V(0)=v,\ N(t) = 2k+1\} = \frac{(2k+1)!}{t^{2k+1}} \sum_{j=1}^{k} \frac{(t-s)^{2k-2j}\,s^{2j}}{j!(j+1)!(2k-2j)!\,2^{2j+1}}\dif s,
\end{align}
\begin{align}\label{tpp0Npari}
P\{F_0 \in \dif s\ |\ V(0)=v,\ N(t) = 2k\} = \frac{(2k)!}{t^{2k}} \sum_{j=1}^{k-1} \frac{(t-s)^{2k-1-2j}\,s^{2j}}{j!(j+1)!(2k-1-2j)!\,2^{2j+1}}\dif s.
\end{align}
Furthermore, for $t>0$
\begin{equation}\label{tpp0}
P\{F_0\in \dif t\ |\ V(0)=v\}=\frac{e^{-\lambda t}}{t}I_1(\lambda t)\dif t.
\end{equation}
\end{thm}

Note that all these results are independent of $c$.

\begin{proof}
Without any loss of generality we consider $V(0)=c$.

Probability (\ref{tpp0N1}) follows by observing that if $N(t)=1$ it is necessary that the change of direction occurs at time $s/2$ in order to return to the origin at time $s$.

In order to return in $0$ at time $s$, the motion need at least one change of direction before time $s/2$. Now, with (\ref{relazioneTempiBetaV}) at hand, we immediately obtain
\begin{align*}
P\{F_0 \in \dif s\ |\ &V(0)=c,\ N(t) = n\} \\
&= \int_0^{\frac{s}{2}} P\{F_{ct_1}  \in \dif s-t_1\ |\ V(0)=c,\ N(t-t_1) = n-1\}P\{T_1\in \dif t_1\ |\ N(t) =n\}
\end{align*}
and results (\ref{tpp0Ndispari}) and (\ref{tpp0Npari}) follows by using Theorem \ref{teorematppN+}.

If $s = t$, we can reach the $0$ only if $V(t)=-c$ (we are assuming $V(0)=c$), meaning that an odd number of switches occur up to time $t$. Now, for $k\in\mathbb{N}_0$, we have that 
\begin{equation}
P\{F_0 \in \dif t\ |\ V(0)=v,\ N(t) = 2k+1\} = \binom{2k+1}{k}\frac{\dif t}{t\,2^{2k+1}} =P^-_{2k+1}\{M(t) = 0\}\dif t /t.
\end{equation}
Hence, for $t>0$, we immediately obtain (\ref{tpp0}).
\end{proof}








%
%
%
%


\begin{thebibliography}{99}
\footnotesize

\bibitem{BOT2005}
{\sc Bean, N., O'Reilly, M.M. and Taylor, P.G.} (2005). Hitting probabilities and hitting times for stochastic fluid flows. {\em Stochastic Processes and Their Applications} {\bf 115,} 1530--1556.

\bibitem{BNO2001}
{\sc Beghin, L., Nieddu, L. and Orsingher, E.} (2001). Probabilistic analysis of the telegrapher's process with drift by means of relativistic transformations. {\em Journal of Applied Mathematics and Stochastic Analysis} {\bf 92,} 11--25.

\bibitem{CO2020}
{\sc Cinque, F. and Orsingher, E.} (2020). On the distribution of the maximum of the telegraph process. {\em Theory of Probability and Mathematical Statistics} {\bf 102,} in press.

\bibitem{Dg2010}
{\sc De Gregorio} (2010). Stochastic velocity motions and processes with random time,. {\em Advances in Applied Probability} {\bf 42,} 1028--1056.

\bibitem{DgOS2005}
{\sc De Gregorio, A., Orsingher, E. and Sakhno, L.} (2005). Motions with finite velocity analyzed with order statistics and differential equations. {\em Theory of Probability and Mathematical Statistics} {\bf 71,} 63--79.

\bibitem{Dc2001}
{\sc Di Crescenzo, A.} (2001). On random motions with velocities alternating at Erlang-distributed random times. {\em Advances in Applied Probability} {\bf 33,} 690--701.

\bibitem{DcIMZ}
{\sc Di Crescenzo, A., Iuliano, A., Martinucci, B. and Zacks, S.} (2013). Generalized telegraph process with random jumps. {\em Journal of Applied Probability} {\bf 50,} 450--463.

\bibitem{DcMZ2018}
{\sc Di Crescenzo, A., Martinucci, B. and Zacks, S.} (2018). Telegraph process with elastic boundary at the origin. {\em Methodology and Computing in Applied Probability} {\bf 20,} 333--352.

\bibitem{DcMPZ2020}
{\sc Di Crescenzo, A., Martinucci, B., Paraggio, P. and Zacks, S.} (2020). Some results on the telegraph process confined by two non-standard boundaries. {\em Methodol. Comput. Appl. Probab.} 22 pp.

\bibitem{F1992}
{\sc Foong, S.~K.} (1992). First passage time, maximum displacement and Kac's solution of the telegrapher equation. {\em Phys. Rev.} {\bf A46,} R707--R710.

\bibitem{FK1994}
{\sc Foong, S.~K. and Kanno, S.} (1994). Properties of the telegrapher's random process with or without a trap. {\em Stochastic Processes and their Applications} {\bf 53,} 147--173.

\bibitem{G1951}
{\sc Goldstein, S.} (1951). On diffiusion by discontinuous movements and the telegraph equation. {\em Quart. J. Mech. Appl. Math.} {\bf 4,} 129--156.

\bibitem{H1994}
{\sc Holmes, E.E., Lewis, M.A., Banks, J.E. and Veit, R.R.} (1994). Partial differential equations in ecology: spatial interactions and population dynamics. {\em Ecology} {\bf 75,} 17--29.

\bibitem{K1974}
{\sc Kac, M.} (1974). A stochastic model related to the telegrapher's equation. {\em Rocky Mountain Journal of Math.} {\bf 4,} 497--509.

\bibitem{KR2013}
{\sc Kolesnik, A.~D. and Ratanov, N.} (2013). {\em Telegraph Processes and Option Pricing,} Springer, Heidelberg.

\bibitem{KO2005}
{\sc Kolesnik, A.~D. and Orsingher, E.} (2005). A planar random motion with an infinite number of directions controlled by the damped wave equation. {\em Journal of Applied Probability} {\bf 42,} 1168--1182.

\bibitem{LR2014}
{\sc Lopez, O. and Ratanov, N.} (2014). On the asymmetric telegraph processes. {\em Journal of Applied Probability} {\bf 51,} 569--589.

\bibitem{M2018}
{\sc Malakar, K., Jemseena, V., Kundu, A., Kumar, K.V., Sabhapandit, S., Majumdar, S.N., Redner, S. and Dhar, A.} (2018). Steady-state, relaxation and first-passage properties of a run-and-tumble particle in one-dimension. {\em Journal of Statistical Mechanics} 043215.

\bibitem{O1990}
{\sc Orsingher, E.} (1990). Probability law, flow function, maximum distribution of wave-governed random motions and their connections with Kirchoff's laws. {\em Stochastic Processes and their Applications} {\bf 34,} 49--66.

\bibitem{O1995}
{\sc Orsingher, E.} (1995). Motions with reflecting and absorbing barriers driven by the telegraph equation. {\em Random Operators Stoch. Equat.} {\bf 3,} 9--21.

\bibitem{ODg2007}
{\sc Orsingher, E. and De Gregorio, A.} (2007). Random flights in higher spaces. {\em Journal of Theoretical Probability} {\bf 20,} 769--806.

\bibitem{OK1996}
{\sc Orsingher, E. and Kolesnik, A.~D.} (1996). Exact distribution for a planar random motion model controlled by a fourth-order hypebolic equation. {\em Theory of Probability and its Applications} {\bf 41,} 379--386.

\bibitem{R2007}
{\sc Ratanov, N.} (2007). A jump telegraph model for option pricing. {\em Quantitative Finance} {\bf 7,} 575--583.

\bibitem{SZ2004}
{\sc Stadje, W. and Zacks, S.} (2004). Telegraph processes with random velocities. {\em Journal of Applied Probability} {\bf 41,} 665--678.

\bibitem{Z2004}
{\sc Zacks, S.} (2004). Generalized integrated telegraph processes and the distribution of related random times. {\em Journal of Applied Probability} {\bf 41,} 497--507.


\end{thebibliography}
\end{document}